\documentclass[12pt]{amsart}
\usepackage{amsmath,amssymb,amsthm}
\usepackage{graphicx,caption,tabularx,multirow}
\usepackage[labelformat=simple,labelfont={}]{subcaption}
\usepackage{url}
\usepackage{color}
\usepackage{enumerate}
\usepackage{cleveref}
\usepackage{tikz}
\usepackage[numbers,sort&compress]{natbib}

\usetikzlibrary{shapes}
\usetikzlibrary{arrows}
\usetikzlibrary{positioning}
\usetikzlibrary{matrix}
\usetikzlibrary{decorations.pathreplacing}

\makeatletter
\def\paragraph{\@startsection{paragraph}{4}%
  \z@\z@{-\fontdimen2\font}%
  {\normalfont\bfseries}}
\makeatother

\newtheorem{theorem}{Theorem}
\newtheorem{lemma}[theorem]{Lemma}
\newtheorem{cor}[theorem]{Corollary}
\newtheorem{prop}[theorem]{Proposition}

\crefname{theorem}{Theorem}{Theorems}
\crefname{lemma}{Lemma}{Lemmas}
\crefname{prop}{Proposition}{Propositions}
\crefname{cor}{Corollary}{Corollaries}
\crefname{section}{Section}{Sections}
\crefname{figure}{Figure}{Figures}

\usepackage{times}

\newcommand{\comm}[1]{}

\newcommand{\df}{\textbf}

\newcommand{\ZZ}{{\mathbb Z}}
\DeclareMathOperator{\id}{id}
\DeclareMathOperator{\rev}{rev}

\title{Representing Permutations with Few Moves}

\author[Bereg]{Sergey Bereg}
\address[Sergey Bereg]{University of Texas at Dallas, USA} \email{besp@utdallas.edu}
\author[Holroyd]{Alexander E.~Holroyd}
\address[Alexander E.~Holroyd]{Microsoft Research, USA} \email{holroyd@microsoft.com}
\author[Nachmanson]{\\ Lev Nachmanson}
\address[Lev Nachmanson]{Microsoft Research, USA} \email{levnach@microsoft.com}
\author[Pupyrev]{Sergey Pupyrev}
\address[Sergey Pupyrev]{University of Arizona, USA} \email{spupyrev@gmail.com}

\date{11 August 2015}

\begin{document}

\begin{abstract}
Consider a finite sequence of permutations of the elements $1,\ldots,n$, with
the property that each element changes its position by at most $1$ from any
permutation to the next.  We call such a sequence a \textit{tangle}, and we
define a \textit{move} of element $i$ to be a maximal subsequence of at least
two consecutive permutations during which its positions form an arithmetic
progression of common difference $+1$ or $-1$.  We prove that for any initial
and final permutations, there is a tangle connecting them in which each
element makes at most $5$ moves, and another in which the total number of
moves is at most $4n$.  On the other hand, there exist permutations that
require at least $3$ moves for some element, and at least $2n-2$ moves in
total.  If we further require that every pair of elements exchange positions
at most once, then any two permutations can be connected by a tangle with at
most $O(\log n)$ moves per element, but we do not know whether this can be
reduced to $O(1)$ per element, or to $O(n)$ in total.  A key tool is the
introduction of certain restricted classes of tangle that perform
pattern-avoiding permutations.
\end{abstract}

\maketitle

\section{Introduction}

Let $S_n$ be the symmetric group of permutations $\pi=[\pi(1),\ldots,\pi(n)]$
on $\{1,\ldots,n\}$, with composition defined via
$(\pi\cdot\rho)(i)=\pi(\rho(i))$.  It is natural to
represent a permutation $\pi$ as a composition of simpler permutations.
Define the \df{swap} $s(i)$ to be the permutation $[1,\ldots,i+1,i,\ldots,n]$
that interchanges $i$ and $i+1$.  We call two permutations $\pi$ and $\rho$
\df{adjacent} if they are related by a collection of non-overlapping swaps,
i.e.\ if $\rho=\pi\cdot s(p_1)\cdots s(p_k)$ where $|p_i-p_j|\geq 2$ for
$i\neq j$.  Equivalently, $\pi$ and $\rho$ are adjacent if
$|\pi^{-1}(i)-\rho^{-1}(i)|\leq 1$ for every $i$.
A \df{tangle} is a finite sequence of permutations in which each
consecutive pair is adjacent.  If a tangle $T$ starts with the identity
permutation $\id=[1,\ldots,n]$ and ends with $\pi$, we say that $T$
\df{performs} $\pi$.

It is straightforward to see that for any permutation $\pi$
there is some tangle that performs $\pi$.  Our goal is to
find tangles with simple and elegant structure.  We may
visualize a tangle as follows.  Consider the sequence of
permutations written in one-line notation
$\pi=[\pi(1),\ldots,\pi(n)]$ in a column from top to bottom
as in \cref{paths}, with equal horizontal and
vertical spacings between symbols.  Then, for each
$i=1,\ldots,n$, draw a polygonal path connecting all
occurrences of the number $i$, from top to bottom, as in
the figure.  The path corresponding to element
$i$ is called \df{path} $i$.  Each line segment of a path
is either vertical or at an angle of $\pm 45^\circ$ to the
vertical.  We call a maximal non-vertical line segment of a
path a \df{move}.  Thus, a move corresponds to a maximal
sequence of swaps $s(p_i)$ that occur between the adjacent
elements in some interval of permutations of the tangle,
and with their locations $p_i$ forming an arithmetic
progression with common difference $\pm 1$.  See
\cref{moves}.  It is convenient to illustrate the structure by shading the area
occupied by swaps, as in \cref{shading}.
Our focus is on minimizing moves among tangles that perform a given permutation.
\begin{figure}
\centering
\begin{subfigure}[t]{.3\textwidth}
\centering
\begin{tikzpicture}[x=14pt,y=14pt]
\large
  \matrix (T) [matrix of math nodes,column sep={14pt,between origins},row
    sep={14pt,between origins}] {
  1&2&3&4&5&6\\
  1&2&3&4&5&6\\
  1&2&4&3&6&5\\
  1&4&2&6&3&5\\
  1&4&2&6&5&3\\
  1&4&2&5&6&3\\
  1&4&2&5&6&3\\
 };
  \draw[red,very thick] (T-1-1.center)
  --++(0,-1)--++(0,-1)--++(0,-1)--++(0,-1)--++(0,-2);
  \draw[orange,very thick] (T-1-2.center)
  --++(0,-1)--++(0,-1)--++(1,-1)--++(0,-1)--++(0,-2);
  \draw[yellow!70!black,very thick] (T-1-3.center)
  --++(0,-1)--++(1,-1)--++(1,-1)--++(1,-1)--++(0,-2);
  \draw[green!70!black,very thick] (T-1-4.center)
  --++(0,-1)--++(-1,-1)--++(-1,-1)--++(0,-1)--++(0,-2);
  \draw[blue,very thick] (T-1-5.center)
  --++(0,-1)--++(1,-1)--++(0,-1)--++(-1,-1)--++(-1,-1)--++(0,-1);
  \draw[purple!50!blue,very thick] (T-1-6.center)
  --++(0,-1)--++(-1,-1)--++(-1,-1)--++(0,-1)--++(1,-1)--++(0,-1);
\end{tikzpicture}
\caption{Permutations and paths.}
\label{paths}
\end{subfigure}
\hfill
\begin{subfigure}[t]{.3\textwidth}
\centering
\tikzset{blob/.style={shape=circle,draw,thick,fill=none,inner sep=3pt}}
\begin{tikzpicture}[x=14pt,y=14pt]
  \draw[red,thick] (1,0)node[above]{$1$}
  --++(0,-1)--++(0,-1)--++(0,-1)--++(0,-1)--++(0,-2)node[below]{$1$};
  \draw[orange,thick] (2,0)node[above]{$2$}
  --++(0,-1)--++(0,-1)--++(1,-1)--++(0,-1)--++(0,-2)node[below]{$2$};
  \draw[yellow!70!black,thick] (3,0)node[above]{$3$}
  --++(0,-1)--++(1,-1)--++(1,-1)--++(1,-1)--++(0,-2)node[below]{$3$};
  \draw[green!70!black,thick] (4,0)node[above]{$4$}
  --++(0,-1)--++(-1,-1)--++(-1,-1)--++(0,-1)--++(0,-2)node[below]{$4$};
  \draw[blue,thick] (5,0)node[above]{$5$}
  --++(0,-1)--++(1,-1)--++(0,-1)--++(-1,-1)--++(-1,-1)--++(0,-1)node[below]{$5$};
  \draw[purple!50!blue,thick] (6,0)node[above]{$6$}
  --++(0,-1)--++(-1,-1)--++(-1,-1)--++(0,-1)--++(1,-1)--++(0,-1)node[below]{$6$};
  \draw[line width=2.5pt,orange,every node/.style=blob] (2,-2)node{}--++(1,-1)node{};
  \draw[line width=2.5pt,green!70!black,every node/.style=blob] (4,-1)node{}--++(-2,-2)node{};
  \draw[line width=2.5pt,yellow!70!black,every node/.style=blob] (3,-1)node{}--++(3,-3)node{};
  \draw[line width=2.5pt,blue,every node/.style=blob]
  (5,-1)node{}--++(1,-1)node{} (6,-3)node{}--++(-2,-2)node{};
  \draw[line width=2.5pt,purple!50!blue,every node/.style=blob]
  (6,-1)node{}--++(-2,-2)node{} (4,-4)node{}--++(1,-1)node{};
\end{tikzpicture}
\caption{Moves (thickened lines) and corners (circles).}
\label{moves}
\end{subfigure}
\hfill
\begin{subfigure}[t]{.3\textwidth}
\centering
\begin{tikzpicture}[x=14pt,y=14pt]
  \fill[blue!20] (3.5,-.5)--++(-2,-2)--++(1,-1)--++(1,1)--++(1,-1)
  --++(-1,-1)--++(1,-1)--++(1,1)
  --++(1,1)--++(-1,1)--++(1,1)--++(-1,1)--++(-1,-1)--cycle;
  \draw[thick] (1,0)node[above]{$1$}
  --++(0,-1)--++(0,-1)--++(0,-1)--++(0,-1)--++(0,-2)node[below]{$1$};
  \draw[thick] (2,0)node[above]{$2$}
  --++(0,-1)--++(0,-1)--++(1,-1)--++(0,-1)--++(0,-2)node[below]{$2$};
  \draw[thick] (3,0)node[above]{$3$}
  --++(0,-1)--++(1,-1)--++(1,-1)--++(1,-1)--++(0,-2)node[below]{$3$};
  \draw[thick] (4,0)node[above]{$4$}
  --++(0,-1)--++(-1,-1)--++(-1,-1)--++(0,-1)--++(0,-2)node[below]{$4$};
  \draw[thick] (5,0)node[above]{$5$}
  --++(0,-1)--++(1,-1)--++(0,-1)--++(-1,-1)--++(-1,-1)--++(0,-1)node[below]{$5$};
  \draw[thick] (6,0)node[above]{$6$}
  --++(0,-1)--++(-1,-1)--++(-1,-1)--++(0,-1)--++(1,-1)--++(0,-1)node[below]{$6$};
\end{tikzpicture}
\caption{Shading the swaps.}
\label{shading}
\end{subfigure}
\caption{A tangle performing the permutation
$\pi=[1,4,2,5,6,3]$, with $7$ moves.}
\end{figure}
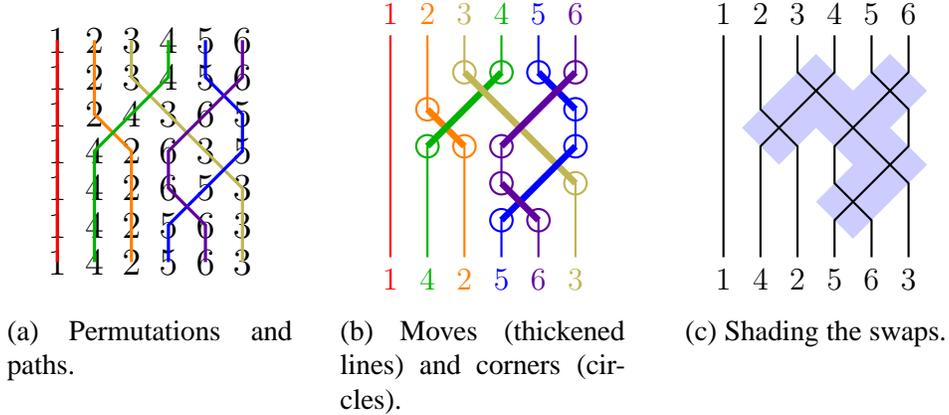

Our first main result is that any permutation can be performed by a tangle
with a bounded number of moves per path (and therefore $O(n)$ moves in total
as $n\to\infty$).  In contrast, various natural greedy algorithms for
constructing a tangle (including one proposed in \cite{w-nrsil-91}) require
$\Omega(n^2)$ moves in total in the worst case.  (See \cref{greedy} for
examples.)

\begin{theorem}\label{fish}
For any permutation $\pi\in S_n$, there is a tangle performing $\pi$ that has
at most $5$ moves in each path.
\end{theorem}

Shifting our attention to {\em total} moves, we can reduce
the constant from $5$ to $4$.

\begin{theorem}\label{8n}
For any permutation $\pi\in S_n$, there is a tangle performing $\pi$ that has
at most $4n$ moves in total.
\end{theorem}

On the other hand, for all sufficiently large $n$ there are permutations that
require at least $3$ moves in some path, and permutations that require at
least $2n-2$ moves in total.  (The latter is easily seen to hold for the
reverse permutation $[n,n-1,\ldots,1]$, while the former apparently requires
a quite involved argument -- see \cref{3-moves}). It is an open problem to
close the gap between the bounds $3$ and $5$ for moves per path, and between
$2n-2$ and $4n$ for total moves.

\begin{figure}
\centering
\begin{minipage}{.49\textwidth}
\includegraphics[width=\textwidth]{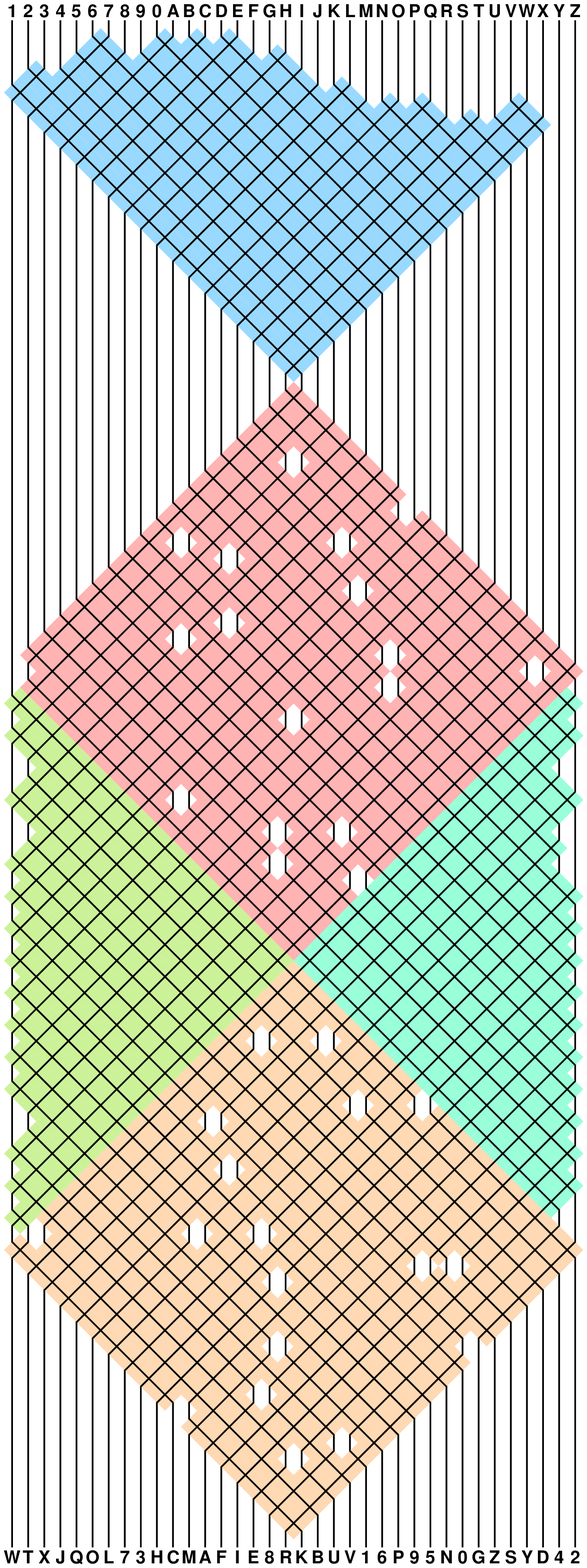}
\subcaption{At most $5$ moves per path (\cref{fish}).}
\label{f:fish}
\end{minipage}
\hfill
\begin{minipage}{.49\textwidth}
{  \includegraphics[width=\textwidth]{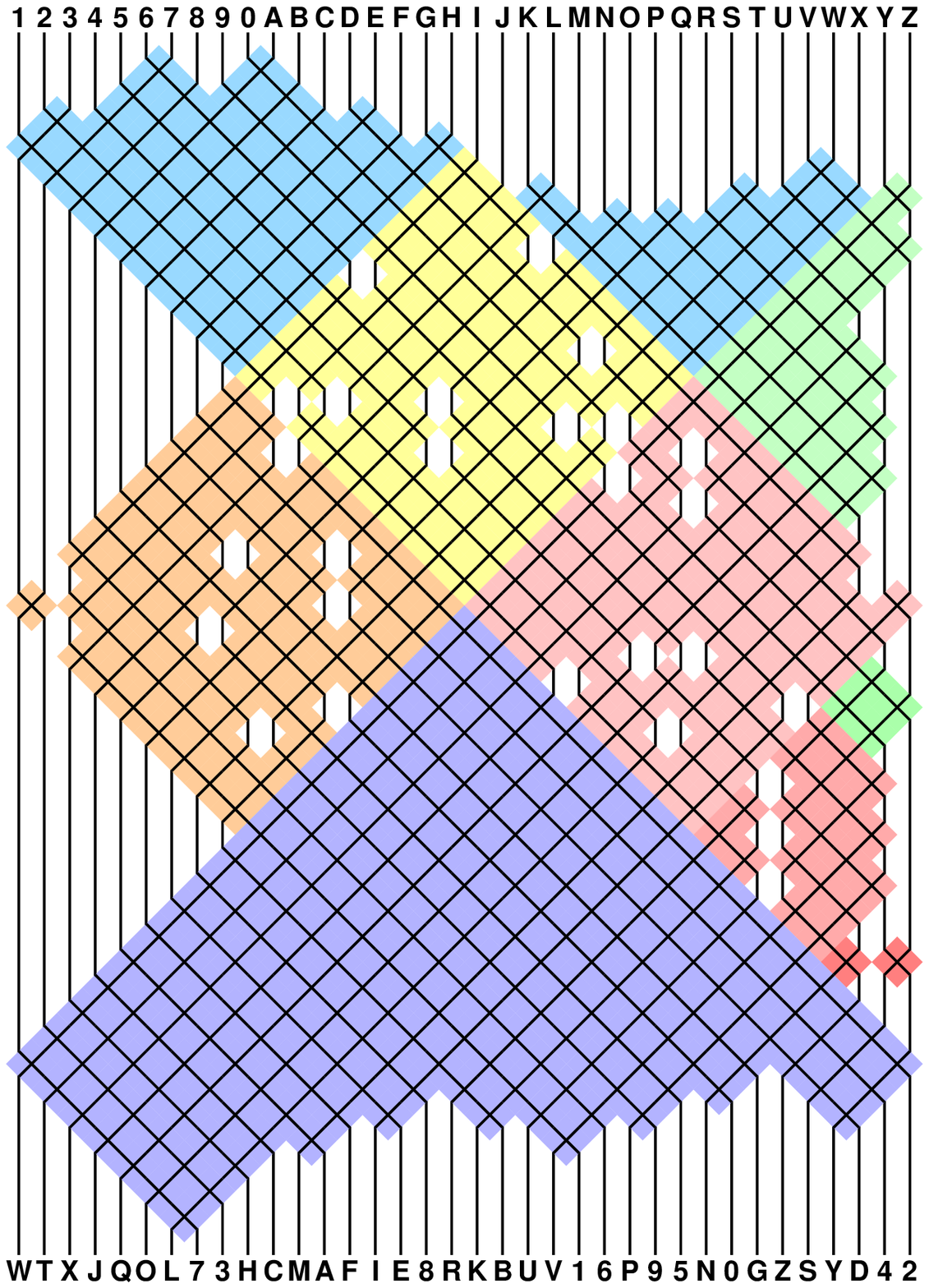}
  \vspace{-.25in}\subcaption{At most $4n$ moves (\cref{8n}).}
  \label{f:8n}}
  \vspace{.1in}
{  \includegraphics[width=\textwidth]{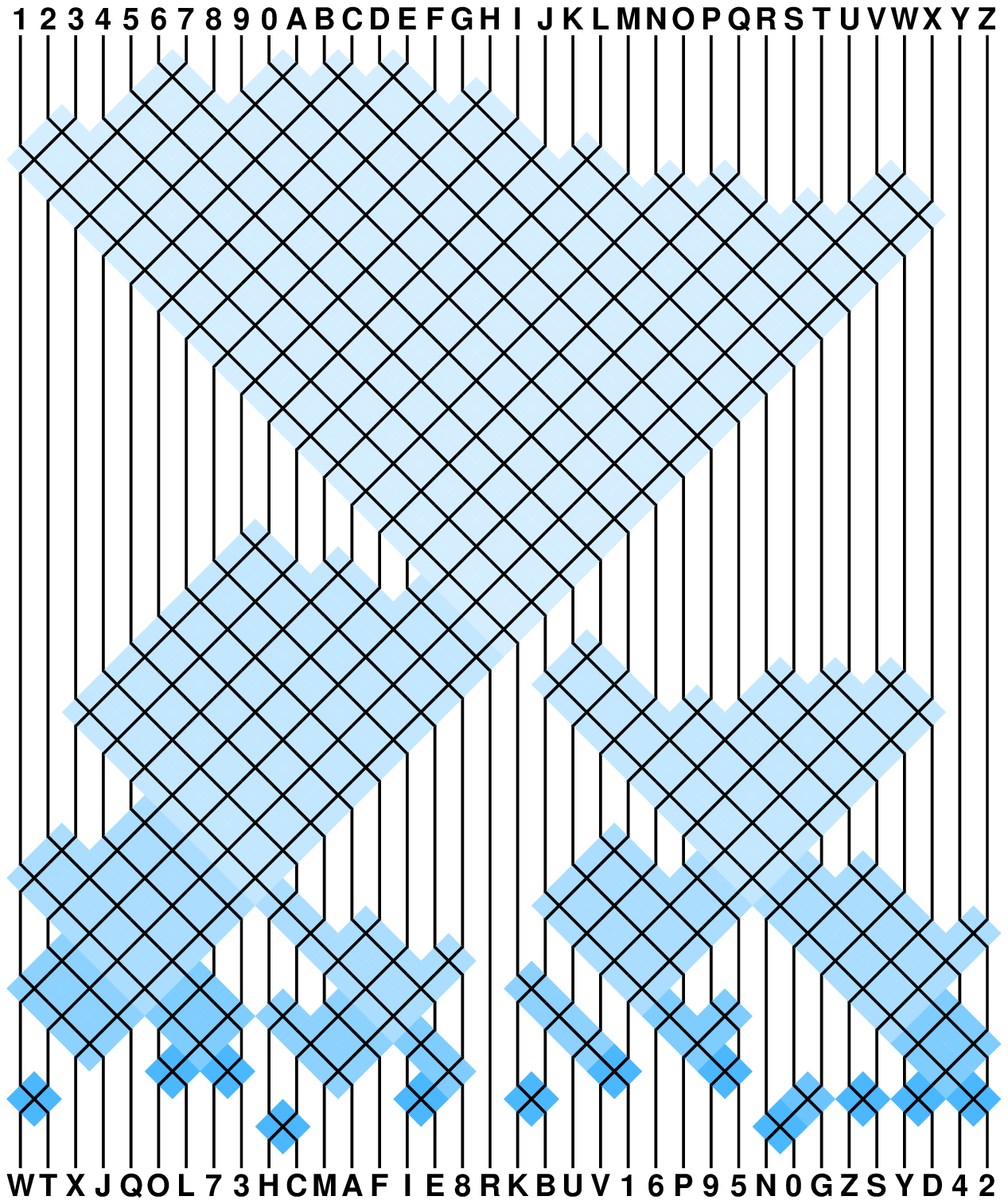}
  \subcaption{Minimum crossings, and
   at most $\lceil\log_2 n\rceil$ moves per path (\cref{log}).}
  \label{f:log}}
\end{minipage}
\caption{Examples of the tangles corresponding to the main results.
Shading is added to illustrate the structure.
}
\label{examples-col}
\end{figure}

\cref{f:fish,f:8n} give examples of the constructions behind
\cref{fish,8n}.  The
tangles will be constructed by combining various
``gadgets'' -- smaller tangles that are capable of
performing permutations in certain restricted classes.
Specifically, we will consider gadgets that perform (and
are in bijective correspondence with) Grassmannian,
$321$-avoiding, $213$-avoiding, and $132$-avoiding
permutations.

Despite the relatively small numbers of moves, the tangles illustrated in
\cref{f:fish,f:8n} arguably have some undesirable features, which we discuss next.
Firstly, they have many ``holes'' -- small internal regions containing no
swaps, shown unshaded in the figures. Secondly, a given pair of paths may
cross multiple times. We will show that some version of the first issue is
unavoidable if the number of moves is to be linear in $n$. On the other hand,
we do not know whether the second issue can be avoided.

Rather than holes, it will be convenient to work with a slightly different
notion, to be defined next.  First we observe that counting moves is essentially
equivalent to counting corners (see also \cite{bhnp}). A \df{corner} is a vertex
of a path, at which its direction changes between any two of the three
possible directions.  Assume that a tangle has its initial and final
permutations repeated at least once, so that each path starts and ends with a
vertical segment.  In addition, count ``double corners'' (at which a path
changes from one non-vertical direction to the other) with multiplicity $2$.
With these conventions, the number of corners in a path equals twice the
number of moves.

In our geometric interpretation of a tangle, we think of the swaps as located
at the elements of the integer lattice $\ZZ^2$.  Therefore, the elements of
the permutations, and thus also the corners, are located at elements of the
shifted lattice $(\ZZ+\tfrac 12)^2$.  Specifically, take the $i$th element
$\pi_t(i)$ of the $t$th permutation $\pi_t$ in the tangle to be located at
the point $(i-\tfrac12,t-\tfrac12)$, where the first coordinate increases
from left to right, and the second coordinate increases from top to bottom.

Given a tangle $T$, consider the graph whose vertices are the corners of $T$,
and with an edge between two corners if their locations are within
$\ell^\infty$-distance $1$. We call the connected components of this graph
\df{clusters}.   (See \cref{f:clusters}.)  The idea is that clusters
generalize the notion of holes discussed above.  Our next result implies
that, as $n\to\infty$, for some (in fact, almost all) permutations, if a
tangle has only $O(n)$ corners (equivalently, $O(n)$ moves) then it must have
at least $\Omega(n)$ clusters. Indeed, $o(n)$ clusters necessitates $\Omega(n
\log n)$ corners. The proof will use a counting argument.

\begin{theorem}\label{counting}
Let $\theta\in(0,\tfrac12)$ and suppose $n>\theta^{-8/\theta}$.  For at
least a proportion $1-e^{-n}$ of the permutations $\pi\in S_n$, any tangle
performing $\pi$ has either at least $(\tfrac12-\theta)n$ clusters or at
least $ \tfrac16 \theta n \log n$ corners.
\end{theorem}

We now turn to the second issue raised above.  We call a tangle \df{simple}
if each pair of paths has at most one crossing.  It is again easy to see that
every permutation admits a simple tangle.  In a simple tangle performing a
permutation $\pi$, paths $\pi(i)$ and $\pi(j)$ cross each other if and only
if $(\pi(i),\pi(j))$ is an \df{inversion} of $\pi$, i.e.\ $i<j$ and
$\pi(i)>\pi(j)$.

The article \cite{bhnp} by the current authors characterizes a class of
permutations for which there exist {\em simple} tangles that have the minimum
moves among {\em all} tangles.  However, there exist permutations that
require strictly more moves for a simple tangle than for a general tangle.  Again,
see \cite{bhnp} for details.

In contrast with the case of general tangles discussed earlier, our upper and
lower bounds for numbers of moves in simple tangles are rather far apart:
$O(n\log n)$ and $\Omega(n)$ respectively as $n\to\infty$.  Closing this gap
is our principal open problem.

\begin{prop}\label{log}
For any permutation $\pi\in S_n$, there is a simple tangle performing $\pi$
that has at most $\lceil \log_2 n\rceil$ moves in each path.
\end{prop}

\begin{prop}
\label{lower} For every $n\geq 1$,  there is a permutation $\pi \in S_n$ such
that any simple tangle that performs it has at least $3n - c\sqrt {n}$ moves,
where $c>0$ is an absolute constant.
\end{prop}

While our focus is on moves, one can attempt to optimize other aspects of a
tangle.  For instance, we may define the \df{depth} of a tangle to be the
length of the sequence of permutations comprising it (including the final
permutation but not the initial one, say).  It is not difficult to check that
any $\pi\in S_n$ can be performed by some tangle of depth at most $n-1$ for
even $n$ and at most $n$ for odd $n$ (and these bounds are optimal; they are
attained by the reverse permutation). Our constructions for
\cref{fish,8n,log} perform reasonably well in this regard, having depths at
most $3n$, $7n/4$ and $3n/2$ respectively.

\subsection*{Background}

Further material on tangles and moves appears in a companion paper
\cite{bhnp} by the current authors.   The main result of \cite{bhnp} is a
surprisingly complex characterization of the set of permutations that can be
performed by a simple tangle in which each path has at most one move in each
direction, together with a polynomial-time algorithm for recognizing such a
permutation and constructing the tangle. (In particular, this set turns out
to include every permutation in $S_6$, but no permutation containing the
pattern $7324651$.) Tangles and related objects have been studied in several
settings by other authors, although the problem of minimizing moves (or
corners) does not appear to have been considered prior to \cite{bhnp}.

Wang in~\cite{w-nrsil-91} considered essentially the same
notion in the context of VLSI design for integrated
circuits. However, the research in~\cite{w-nrsil-91}
targets, in our terminology, the depth of a tangle, and the
total length of the paths. The algorithm suggested by Wang
produces tangles with $O(n^2)$ moves for some permutations.
%see Appendix~\ref{sect:wang}

In algebraic combinatorics, Schubert polynomials can be encoded as sums over
diagrams called RC-graphs or pipe dreams \cite{pipe1,pipe2}, which may be
interpreted as tangles of a certain type.  Specifically, an RC-diagram
corresponds via a $45^\circ$ rotation to a simple tangle whose swaps are
restricted to odd locations in a triangular region (the same region as our
``reflector gadget'' in \cref{ss:reflector}). Reduced words for permutations
are extensively studied; see e.g.\ \cite{knuth3,Angel07,young,saga}.  In our
terminology, a reduced word is a simple tangle with only one swap between
consecutive permutations.

\sloppypar Decomposition of permutations into nearest-neighbour
transpositions was considered in the context of permuting machines and
pattern-restricted classes of permutations~\cite{Albert07}. In our
terminology, Albert~et.~al.~\cite{Albert07} proved that it is possible to
check in polynomial time whether for a given permutation there exists a
tangle of depth $k$, for a given $k$.  Tangles and the associated
visualizations also appear in sorting networks \cite{aks,knuth3}, in
arrangements of pseudolines~\cite{Felsner96}, and in the context of change
ringing (English-style church bell ringing) \cite{white}.  In the terminology
of change ringing, a tangle with minimum corners is a ``link method with
minimum changes of direction''; each permutation represents an order of
ringing the bells, and a corner requires a ringer to change the speed of
their bell, which involves extra physical effort.  Also related is the
problem of decomposing a permutation into the minimum number of block
transpositions -- see \cite{bridge-hand}.

Tangles appear naturally as a sub-problem in the context of graph-drawing,
and this was our original motivation for the problems considered here.  In
order to simplify a visualization of a large graph, it is sometimes
advantageous to ``bundle'' sets of nearby edges together
\cite{sergey10,pnbh-ifmpg-11}.  Since the edges may be required to appear in
different orders at the two ends of a bundle, they must be permuted along its
length, and it is desirable to do this in a helpful and visually appealing
way. Paths with few moves (or few corners) tend to be easy to follow.

With practical applications in mind, it is worth noting that the tangles
resulting from our constructions can often be improved slightly by local
modifications. For example, in \cref{f:fish}, one may eliminate the two swaps
where the tail and body of the ``fish'' meet, reducing the depth; in
\cref{f:8n}, the isolated swap in the middle of the leftmost column may be
moved upward to meet the swaps at the top, eliminating a move. Such
modifications may be iterated, but will not improve the worst case asymptotic
performance of the constructions.

\begin{figure}
\centering
\begin{minipage}{.45\textwidth}
  \includegraphics[width=\textwidth]{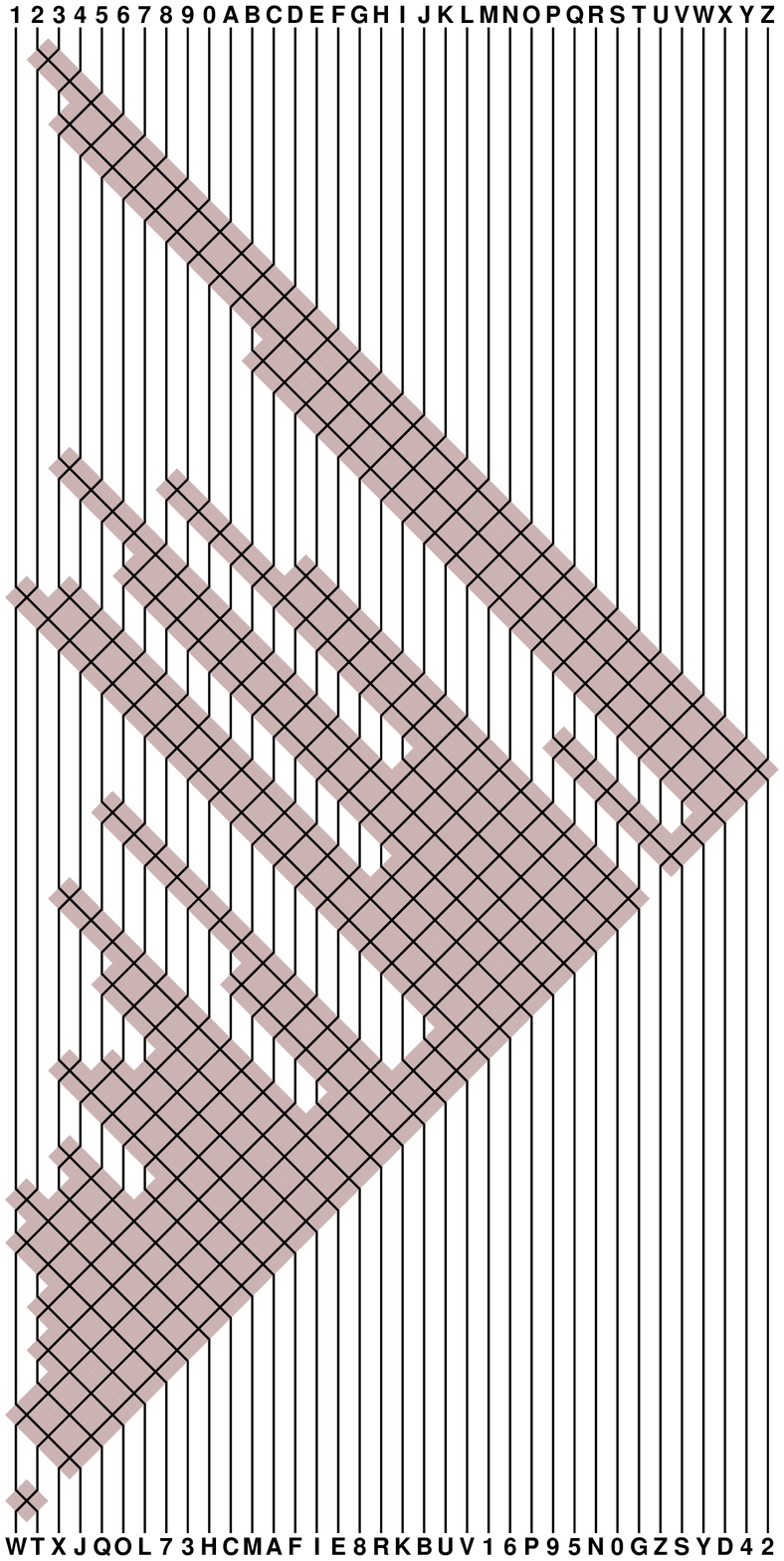}
  \subcaption{Bubble sort variant: use one R-move to route each path to its correct position, starting from the rightmost, $\pi(n)$.  Path $i$ may have $\Omega(i)$ L-moves.}
  \label{bubble}
\end{minipage}
\hfill
\begin{minipage}{.45\textwidth}
\vfill\includegraphics[width=\textwidth]{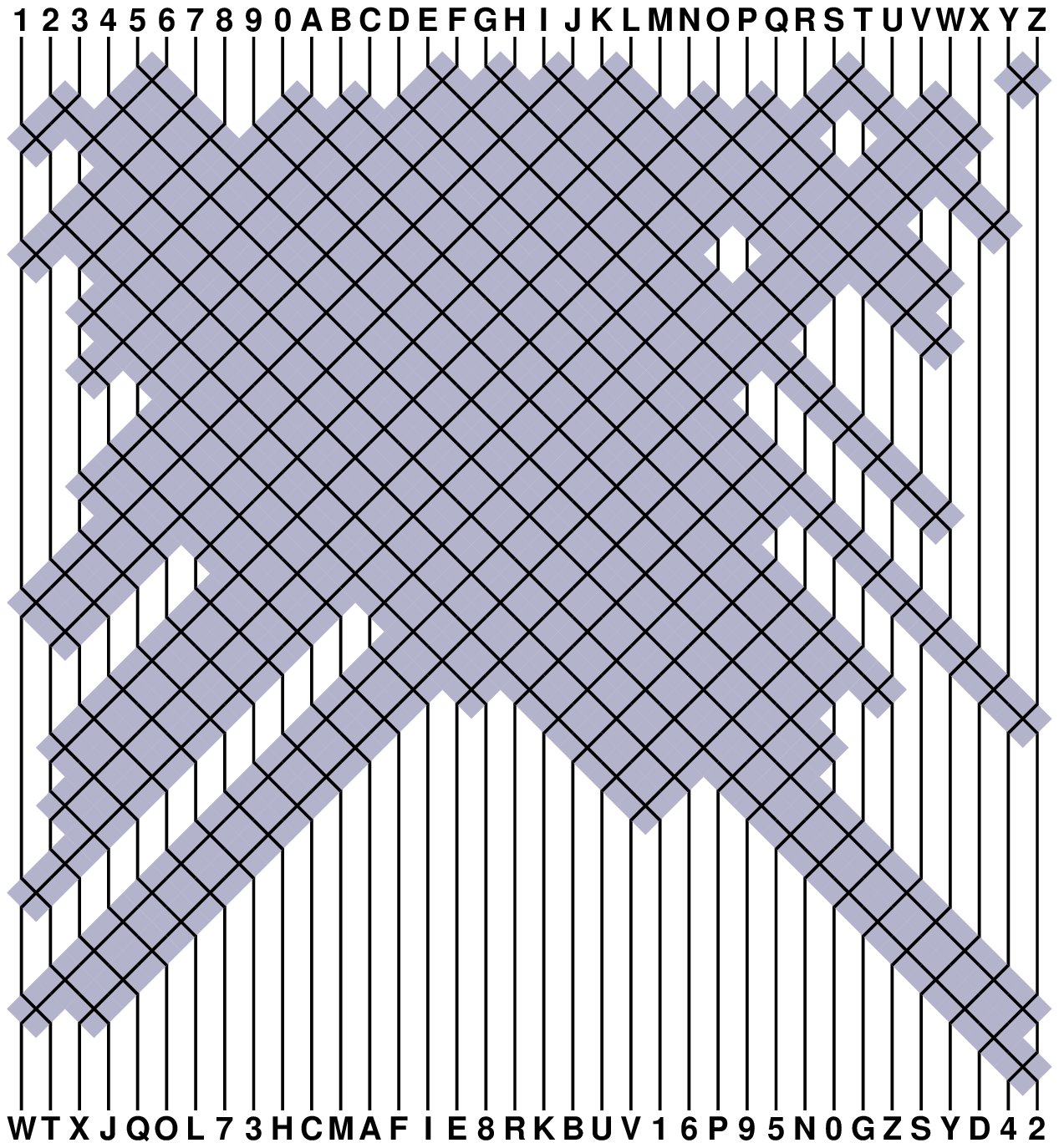}\vfill
  \subcaption{Odd-even sort: at alternate steps, apply swaps in all odd positions, or all even positions, wherever the two elements form an inversion.}
\end{minipage}
\caption{Tangles constructed according to two natural greedy algorithms.
Both require $\Omega(n^2)$ moves in the worst case as $n\to\infty$.}
\label{greedy}
\end{figure}

\subsection*{Further notation and conventions}

As mentioned above, it is convenient to consider a tangle in terms of its
swaps, and we think of the swaps as located at elements of the integer
lattice $\ZZ^2$.  If $\pi_t,\pi_{t+1}\in S_n$ are two consecutive
permutations in a tangle, and they are related by non-overlapping swaps thus:
$\pi_t\cdot s(p_1)\cdots s(p_k)=\pi_{t+1}$, then we say that the tangle has
swaps at \df{locations} $(p_1,t),\ldots,(p_k,t)$.  The first coordinate is
sometimes called position, and increases from left (West) to right (East)
(from $1$ to $n-1$); the second coordinate is called time, and increases from
top (North) to bottom (South).  If a tangle consists of permutations in $S_n$
then we sometimes call $n$ the \df{width} of the tangle.

We identify two tangles if they have the same set of swap locations; thus, we
consider the tangle with permutations $\pi_1,\ldots,\pi_t$ to be the same as
that with permutations $\gamma\cdot\pi_1,\ldots,\gamma\cdot\pi_t$, for any
permutation $\gamma$.  In particular, a tangle that performs a permutation
$\pi$ may be equivalently be considered as starting at $\pi^{-1}$ and ending
at $\id$, thus ``sorting'' $\pi^{-1}$.  The latter convention was adopted in
\cite{bhnp}.  It will also be useful to allow times of swaps to take {\em
any} value in $\ZZ$, and to identify two tangles if one is obtained from the
other by adding a constant to all swap times (thus translating it
vertically).

As mentioned earlier, we will construct tangles by
combining smaller tangles (called gadgets), and for this it
will be useful to translate horizontally as well as
vertically.  Thus, let $m<n$ and suppose that $T$ is a
tangle performing $\pi\in S_m$, with its swaps at locations
$S\subset[1,m-1]\times \ZZ$.  Then for integers $a,b$, we
may form a tangle $T' $of size $n$ by placing swaps at the
translated locations $S':=\{(i+a,t+b):(i,t)\in S\}$; this
performs the permutation
$[1,\ldots,a,\pi(a+1),\ldots,\pi(a+m),a+m+1,\ldots,n]$.
Moreover, we may combine several tangles by taking the
union of their sets of swap locations (perhaps after
applying various translations).

A swap location $(x,t)$ is called \df{even} or \df{odd} according to whether
$x+t$ is even or odd.  All the tangles we construct will have their swaps
restricted to locations of one parity.  As indicated above, a convenient way
to highlight the structure of such a
tangle is to draw a shaded $45^\circ$-rotated square centered
at each swap, as in \cref{shading}.  Recall that a move is a maximal
non-vertical segment of a path. We call it an \df{L-move} if it runs in the
North-East to South-West direction, and an \df{R-move} if it runs North-West
to South-East.

Pattern-avoiding permutations will play a key role. (See e.g.~\cite{pattern}
for background.) A \df{pattern} is a permutation $p\in S_m$.  For $n\geq m$,
we say that a permutation $\pi\in S_n$ (or, more generally, a sequence of $n$
distinct real numbers $\pi$) \df{contains} the pattern $p$ if there exist
indices $1\leq i_1<\cdots<i_m\leq n$ such that, for all $1\leq j<k\leq m$, we
have $\pi(i_j)<\pi(i_k)$ if and only if $p(j)<p(k)$.  If $\pi$ does not
contain $p$ then $\pi$ is said to be $p$\df{-avoiding}.

The following concept will also be useful.  For positive integers
$a_1,\ldots,a_k$ with sum $n$, consider the partition of $[1,n]$ into the
intervals $[1,a_1],[a_1+1,a_1+a_2],\ldots,[n-a_k+1,n]$ with these lengths. We
say that a permutation $\pi\in S_n$ is $(a_1,\ldots,a_k)$\df{-split} if it
maps each of these intervals to itself.

\subsection*{Organization of the paper}

In \cref{s:gadgets} below we introduce the gadgets that will be used in our
constructions, and prove their required properties.  \cref{log} and
\cref{fish,8n} are then proved in \cref{s:log,s:fish,s:8n} respectively.  We
prove the bound \cref{counting} in \cref{s:clusters}, via a combinatorial
argument. In contrast, the lower bound in \cref{lower} and the fact that some
permutations require $3$ moves in some path (\cref{3-moves}) are proved by
explicitly exhibiting suitable permutations, in
\cref{s:simple-lower,s:path-lower} respectively.  Although the permutations
in question are very easy to describe, the proofs of both results are
surprisingly delicate.

\section{Gadgets}
\label{s:gadgets}

In this section we introduce the gadgets that will be used to prove
\cref{fish,8n}.  They come in three main categories, with several variants in
each.

\subsection{Splitter and Merger}

Our first gadget comes in two variant forms, which are
reflections of each other about a horizontal axis. A
\df{splitter} gadget has swaps at locations
$$(i-j+a,-i-j),$$
for all $j \geq 0$ and $0\leq i \leq b(j)$, where $a$ is an even integer, and
$b$ is a non-decreasing integer-valued function of bounded support. Thus, a
splitter consists of swaps at all even locations in a region bounded below by
two line segments running South-East and North-East, and bounded above by an
interface comprising any sequence of North-East and South-East segments. See
\cref{f:splitter} for an example. The idea is that it separates the paths
into two arbitrary sets, and places them on the left and right sides while
maintaining the relative order within each set.
\begin{figure}
\centering
\begin{subfigure}{0.35\textwidth}\centering
\includegraphics[width=\textwidth]{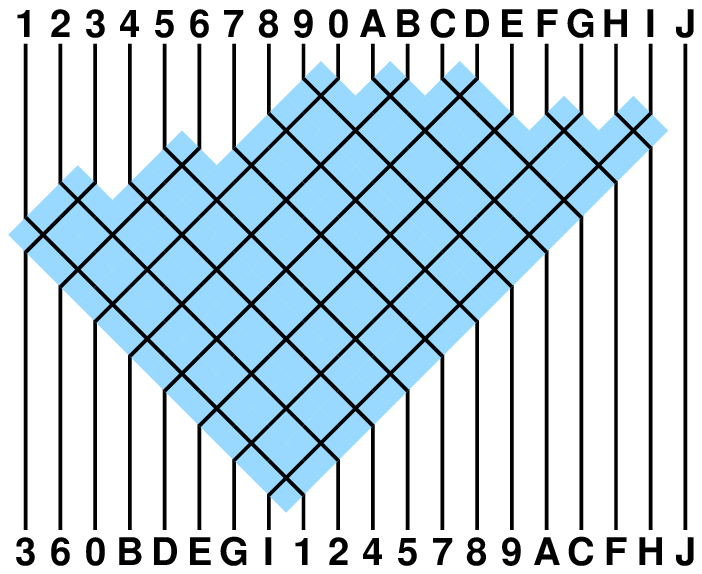}
\subcaption{Splitter gadget.}
\label{f:splitter}
\end{subfigure}\hspace{.1\textwidth}
\begin{subfigure}{0.35\textwidth}\centering
\includegraphics[width=\textwidth]{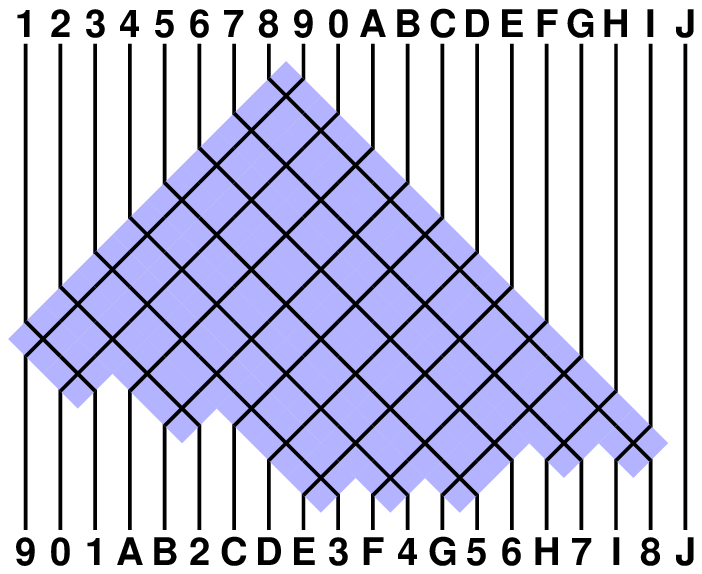}
\subcaption{Merger gadget.}
\label{f:merger}
\end{subfigure}
\caption{}
%\caption{A splitter (left); a merger (right).}
\end{figure}

To formalize this: a permutation $\pi=[\pi(1),\ldots,\pi(n)]$ is called
\df{Grassmannian} if it has at most one descent, i.e.\ at most one index $k$
such that $\pi(k)>\pi(k+1)$.

\begin{lemma}\label{splitter}
A permutation can be performed by some splitter if and only if it is
Grassmannian.  Furthermore, the correspondence between splitters and
Grassmannian permutations is bijective.
\end{lemma}

For the purposes of the claimed bijectivity, recall that
two tangles are identified if they have the same set of
swap locations.

\begin{proof}[Proof of \cref{splitter}]
The identity permutation is clearly performed by the
trivial tangle containing no swaps.  Any other Grassmannian
permutation $\pi$ has exactly one descent; say
$\pi(k)>\pi(k+1)$.  We take $a=k$, and
$b(i)=\pi(k-i)-(k-i)$ for $0\leq i \leq k-1$, and $b(i)=0$
for $i\geq k$.  The function $b$ is easily seen to be
non-decreasing.  The lower boundary of the splitter
consists of $k$ steps South-East followed by $n-k$ steps
North-East. The upper boundary also consists of $k$
South-East steps and $n-k$ North-East steps, with the
$\pi(i)$th step being South-East if and only if $i\leq k$.
For $i\leq k$, path $\pi(i)$ makes one L-move, starting at
position $i$ and ending at position $\pi(i)$. For $i> k$,
path $\pi(i)$ similarly makes one R-move. The paths
$\pi(i)$ for $i\leq k$ maintain their order relative to
each other, as do those for $i> k$. See \cref{f:splitter}.  By
similar reasoning, any splitter performs a Grassmannian
permutation.  Since different splitters perform different
permutations, the correspondence is bijective.
\end{proof}

A \df{merger} is obtained by reflecting a splitter about a horizontal axis.
Thus it has swaps at all locations
$$(i-j+a,i+j),$$
for $a$ and $b(\cdot)$ as above.  See \cref{f:merger}.  The corresponding
permutation is the inverse of that performed by the splitter.  A permutation
$\pi$ is the inverse of a Grassmannian permutation if and only if, for some
$k$, the values $1,\ldots,k$ appear in increasing order in the sequence
$\pi=[\pi(1),\ldots,\pi(n)]$, and so do $k+1,\ldots,n$.  The proof of the
following lemma is immediate.

\begin{lemma}\label{merger}
A permutation can be performed by some merger if and only if its inverse is
Grassmannian. Furthermore, this correspondence is bijective.
\end{lemma}

\begin{figure}
\centering
\includegraphics[width=0.35\textwidth]{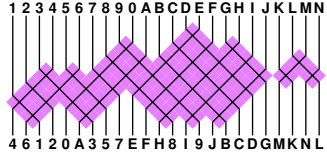}
\caption{A direct tangle, performing a $321$-avoiding permutation.}
\label{f:321}
\end{figure}

Both splitters and mergers are special cases of a more general class of
tangles considered in \cite{bhnp}, called {\em direct} tangles.  A direct
tangle is one in which each path has at most one move.  Modulo a suitable
convention regarding split permutations (in which the parts of the tangle
corresponding to different splitting intervals may be translated vertically),
such a tangle consists of swaps at all even locations within a region bounded
above {\em and} below by interfaces comprising North-East and South-East
segments. See \cref{f:321}.  It is shown in \cite{bhnp} that a permutation
admits a direct tangle if and only if it is $321$-avoiding.  (Grassmannian
permutations and their inverses are indeed $321$-avoiding.) The
correspondence is again bijective.

\subsection{Matrix gadget}

Let $n=2m$ be even, and let $\alpha\in S_m$ be a permutation.  The \df{matrix
gadget} indexed by $\alpha$ consists of swaps at the locations
$$(i+j-1,i-j)$$
for all pairs $i,j\in\{1,\ldots,m\}$ {\em except} those with $\alpha(i)= j$.
In other words, a square angled at $45^\circ$ to the axes is filled with
swaps at all odd locations, except for those locations corresponding to the
support of the (rotated) permutation matrix of $\alpha$.  See \cref{f:matrix}
for an example. The idea is that a matrix gadget performs any given
permutation on one half; the following result says that the effect on the
second half is the inverse permutation.
\begin{figure}
\centering
\begin{subfigure}[t]{.35\textwidth}
\centering
\includegraphics[width=\textwidth]{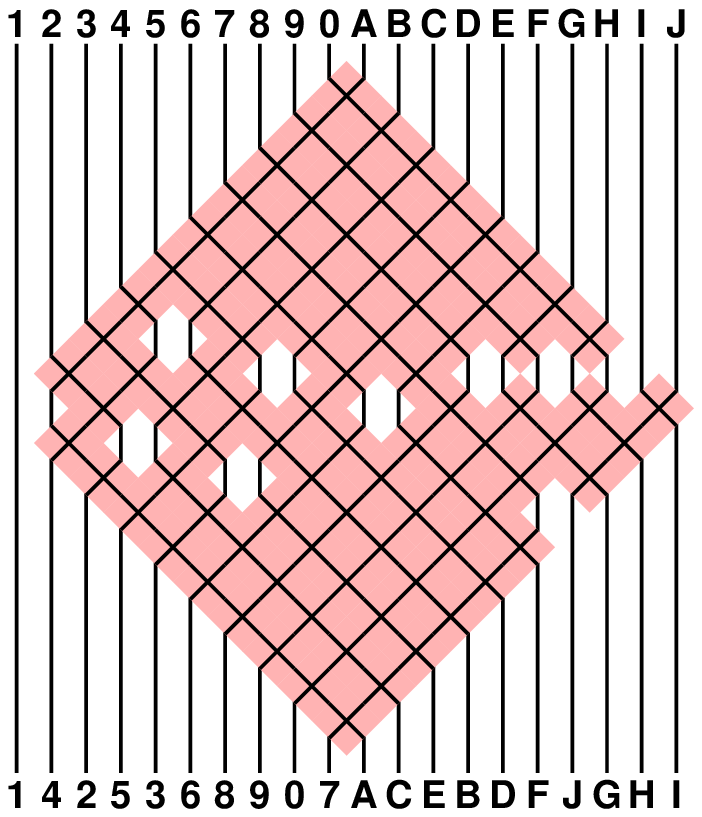}
\caption{Matrix gadget indexed by the permutation $[1,4,2,5,3,6,8,9,0,7]$.}
\label{f:matrix}
\end{subfigure}
\hspace{.1\textwidth}
\begin{subfigure}[t]{.35\textwidth}
\centering
\includegraphics[width=\textwidth]{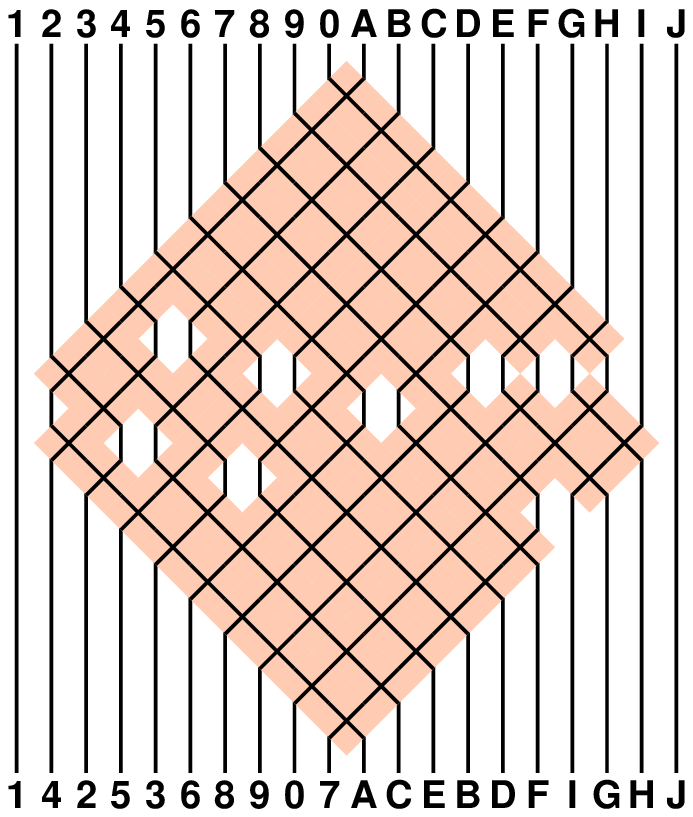}
\caption{The same gadget truncated on the right.}
\label{f:matrix-trunc}
\end{subfigure}
\caption{}
\end{figure}

\begin{lemma}
\label{lm:square} Let $n=2m$ and let $\alpha\in S_m$.  The matrix gadget
indexed by $\alpha$ performs the permutation
$$\bigl[\alpha(1),\alpha(2),\ldots,\alpha(m),\quad
\alpha^{-1}(1)+m, \alpha^{-1}(2)+m, \ldots,\alpha^{-1}(m)+m\bigr]\in S_n.$$
\end{lemma}

\begin{figure}
\centering
\begin{tikzpicture}[scale=.4]
\fill[red!20] (0,0)--(6,6)--(12,0)--(6,-6)--cycle;
\fill[white] (4,-2)--(5,-1)--(6,-2)--(5,-3)--cycle;
\draw[very thick,blue]
(1.5,1.5)node[above left,blue]{$j$}--(4.5,-1.5)--(4.5,-2.5)--(3.5,-3.5)node[below left,blue]{$j$};
\draw[very thick,green!70!black]
(9.5,2.5)node[above right,green!70!black]{$i+m$}--(5.5,-1.5)--(5.5,-2.5)--
(7.5,-4.5)node[below right,green!70!black]{$i+m$};
\draw[dashed,thick,black] (3.5,7)node[above]{$i$}--(3.5,-3.5);
\draw[dashed,thick,black] (1.5,7)node[above]{$j$}--(1.5,1.5);
\draw[dashed,thick,black] (9.5,7)node[above,xshift=8pt]{$i+m$}--(9.5,2.5);
\draw[dashed,thick,black] (7.5,7)node[above,xshift=-8pt]{$j+m$}--(7.5,-4.5);
%\draw[thick,black] (0,8)node[above]{$1$}--(12,8)node[above]{$n$};
\draw[thick,black] (0,7)--(12,7);
\end{tikzpicture}
\caption{A pair of complementary paths in a matrix gadget.
Horizontal positions are indicated along the top line.
}\label{construction-matrix}
\end{figure}
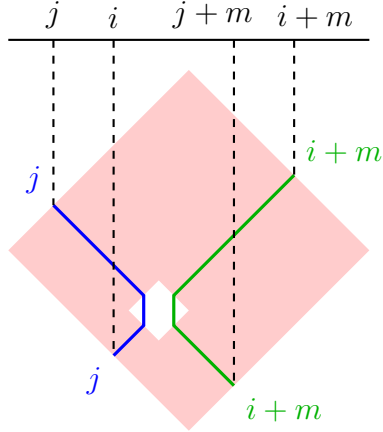
\begin{proof}
This is straightforward to check.   Suppose $\alpha(i)=j$.
Then path $j$ makes an R-move until it encounters the
``omitted swap'' corresponding to the pair $(i,j)$, and
then makes a vertical segment of length $1$ followed by an
L-move, finishing in position $i$. Similarly, path $i+m$
finishes in position $j+m$ after an L-move and an R-move.  See \cref{construction-matrix}.
\end{proof}

The matrix gadget is fundamentally more powerful than our other gadgets, in
the sense that it can perform $(n/2)!$ different permutations in $S_n$,
whereas each the others can only perform $O(c^n)$ permutations for some
constants $c$.  The matrix gadget is the source of the ``holes'' (or, more
generally, clusters) mentioned in the introduction. \cref{counting} reflects
the fact that some such construction is a requirement if we are to have only
linearly many moves.

For some of our constructions, we will need the following variants of the
matrix gadget for odd $n$.  Let $n=2m-1$, and let $\alpha\in S_{m}$.  The
\df{truncated matrix gadget} indexed by $\alpha$ is simply the matrix gadget
of the larger size $2m$ indexed by $\alpha$, but with the rightmost swap (in
location $(2m-1,0)$) omitted (whether or not it is present in the original
matrix gadget).  See \cref{f:matrix-trunc} for an example.  This gadget
performs a permutation of the form
$$\bigl[\alpha(1),\alpha(2),\ldots,\alpha(m),\quad
\ldots\bigr]\in S_{2m-1};$$ i.e.\ $\alpha$ on the $m$ leftmost positions,
 and {\em some} permutation on the $m-1$ rightmost positions.
The precise nature of the permutation on the right will not matter for our
applications.  Similarly, we may truncate a matrix gadget on the left side to
obtain any desired permutation on the rightmost $m$ positions.

Finally, note the following subtle variation.  If $n=2m$ and the index
permutation satisfies $\alpha(1)=1$, then the standard matrix gadget {\em
already} has no swap in the leftmost column.  \cref{f:matrix} is an example.
Therefore, it can also be regarded as a gadget involving only positions
$2,\ldots, 2m$, and performing any desired permutation on the positions
$2,\ldots,m$.  (And it may then be translated one position leftward, for
example).

\subsection{Reflectors}
\label{ss:reflector}

Our final gadget also comes in two complementary forms,
this time related by reflection in a vertical axis. A
\df{right reflector} gadget consists of swaps at locations
$$(i+j+1,j-i)$$
for all $j \geq 0$ and $0\leq i \leq b(j)$, where $b$ is a non-decreasing
integer-valued function of bounded support. Thus, a right reflector consists
of swaps at all odd locations in a region bounded on the left by two line
segments running South-West and South-East, and bounded on the right by an
interface comprising a sequence of South-West and South-East segments.  See
\cref{f:right}.  In this case, this rightmost bounding interface must stay to
the left of the horizontal coordinate $n$.  Therefore it corresponds to a
Dyck path. The idea of the right reflector is that every path starts with an
R-move, then has a vertical segment (now possibly of length greater than
$1$), and then is ``reflected'' back with an L-move.
\begin{figure}
\centering
\begin{subfigure}{.25\textwidth}
\includegraphics[width=\textwidth]{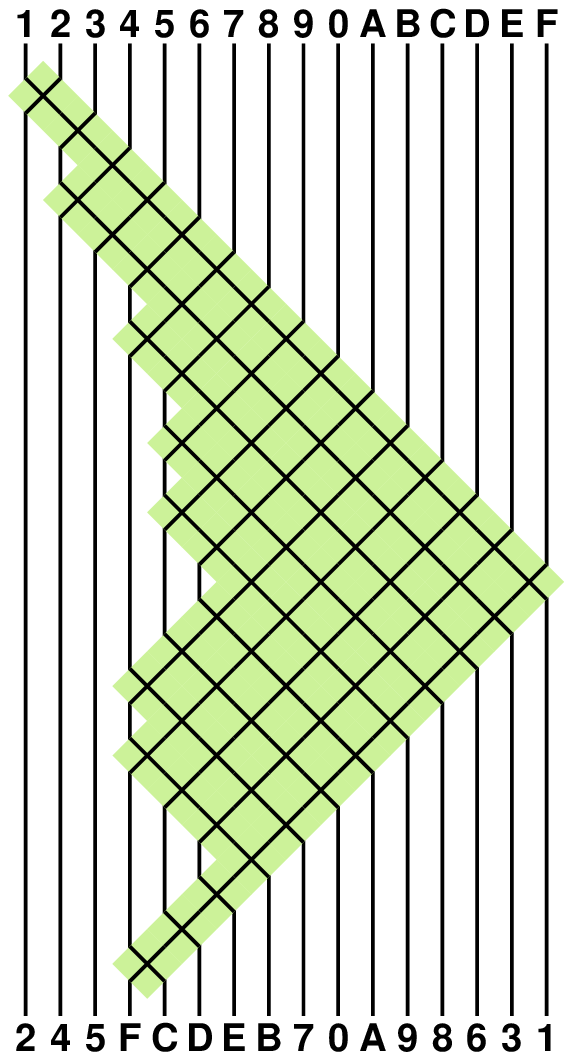}
\caption{A left reflector.}\label{f:left}
\end{subfigure}
\hspace{.1\textwidth}
\begin{subfigure}{.25\textwidth}
\centering
\includegraphics[width=\textwidth]{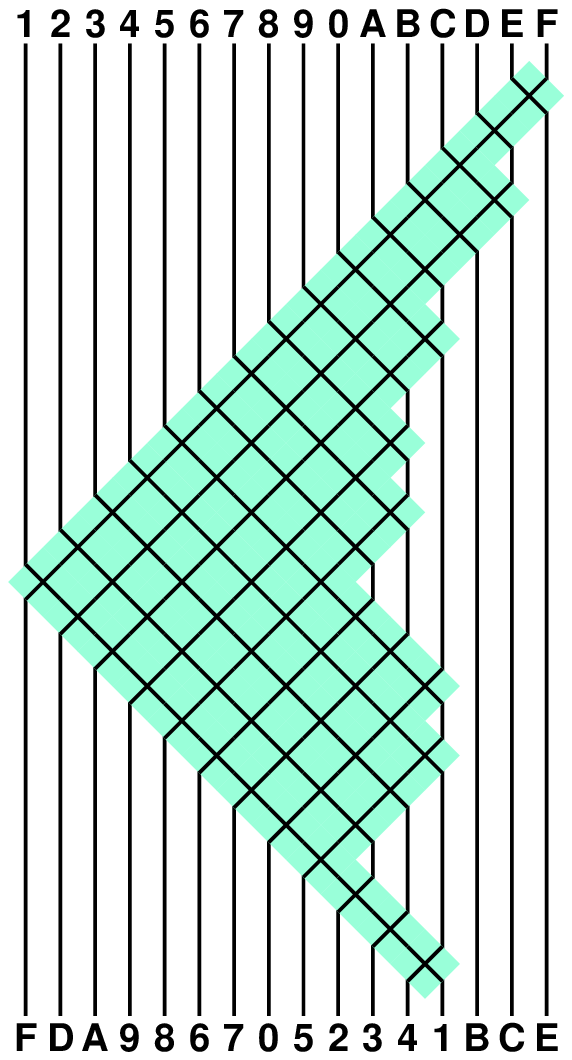}
\caption{A right reflector.}\label{f:right}
\end{subfigure}
\caption{}
\end{figure}

\begin{lemma}\label{r-reflector}
A permutation can be performed by some right-reflector if and only if it is
$132$-avoiding. Furthermore, this correspondence between gadgets and
permutations is bijective.
\end{lemma}

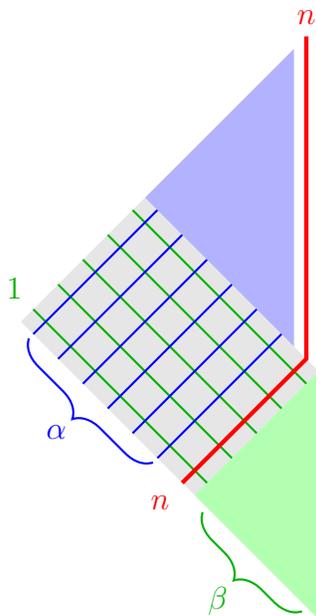
\begin{figure}
\centering
\begin{tikzpicture}[scale=.33]

\fill[black!10] (0,0)--(5,5)--(12,-2)--(7,-7); \fill[blue!30]
(5,5)--(11,11)--(11,-1); \fill[green!30] (7,-7)--(12,-12)--(12,-2);

\foreach \x in {.5,...,4.5}
  \draw[thick,green!70!black] (\x,\x)--(\x+7,\x-7);
\foreach \x in {.5,...,5.5}
  \draw[thick,blue] (\x,-\x)--(\x+5,-\x+5);

\draw[ultra thick, red] (6.5,-6.5)node[below
left]{$n$}--(11.5,-1.5)--(11.5,11.5)node[above]{$n$}; \draw
(.5,.5)node[above left,green!70!black]{$1$};

\draw [thick,blue,decorate,decoration={brace,amplitude=10pt},xshift=-5pt,yshift=-5pt]
(5.5,-5.5) -- (.5,-.5) node [below left=7pt,blue,midway]{$\alpha$};
\draw [thick,green!70!black,decorate,decoration={brace,amplitude=10pt},xshift=-5pt,yshift=-5pt]
(11.5,-11.5) -- (7.5,-7.5) node [below left=7pt,green!70!black,midway]{$\beta$};
\end{tikzpicture}
\caption{Inductive construction of a right reflector. The rectangle is chosen so as
to route path $n$ to its correct location, and the two remaining
triangles are then filled with smaller right reflectors.}
\label{reflector-construction}
\end{figure}

\begin{proof}
We prove the ``if'' direction by induction on $n$.  For $n= 1$, the claim
is clear.  For $n>1$, suppose that $\pi$ is $132$-avoiding.  Consider
the location of element $n$ in $\pi$, and write $\pi=[\alpha,n,\beta]$, where
$\alpha,\beta$ are the sequences of numbers to the left and right of $n$.
Note that $\alpha$ and $\beta$ are both $132$-avoiding.  Also, every element of
$\alpha$ is greater than every element of $\beta$, otherwise we would have a
$132$ pattern including $n$.

We construct a right reflector as shown in \cref{reflector-construction}.
There is a $45^\circ$ rectangle filled with swaps, with one corner at $(1,0)$
and an opposite corner at $(n-1,n-2\pi^{-1}(n))$; path $n$ has a vertical
segment until it hits this rectangle just above its rightmost corner, and
then has an L-move.  (A trivial case is when $\pi^{-1}(n)=n$, the rectangle
is empty, and path $n$ is vertical).  Finally, we use the inductive
hypothesis to insert two strictly smaller right reflectors, which perform the
permutations corresponding to relative orders of $\alpha$ and $\beta$, in the
triangular regions to the North-East and South-East of the rectangle.

We now turn to the ``only if'' direction.  Suppose that a right reflector
gadget $T$ performs a permutation $\pi$.  We first note that $T$ is simple.
Indeed, every path consists of an R-move, then a vertical segment, then an
L-move (where it is possible that one or both of these moves is empty); since
two paths can only cross during the R-move of one and the L-move of the
other, they cannot cross more than once.  Now suppose for a contradiction
that $\pi$ contains a $132$ pattern.  Thus, there exist $u<v<w$ with
$\pi(u)<\pi(w)<\pi(v)$.  Consider the location $x$ of the unique swap between
paths $\pi(v)$ and $\pi(w)$.  By the definition of the right reflector, every
odd location in the $45^\circ$ rectangle with corners $(1,0)$ and $x$
contains a swap.  However, path $\pi(u)$ starts to the left of path $\pi(w)$,
and traverses the entire rectangle during its R-move, and crosses path
$\pi(v)$ at the South-East side of the rectangle. This contradicts
simplicity.

To check bijectivity, since clearly every gadget performs only one
permutation, it is enough to check that the two sets have equal cardinality.
The number of $132$-avoiding permutations in $S_n$ is given by the Catalan
number $C_n$.  A right reflector gadget is encoded by a Dyck path describing
its right boundary.  Therefore the number of them is also $C_n$.  See
e.g.~\cite[Ex.~6.19]{s-ec-99}.
\end{proof}

We remark that the standard Catalan recurrence $C_{n+1}=\sum_{i=0}^n C_i
C_{n-i}$ is implicit in our inductive construction above.  Arguments similar
to ours appear in the context of \textit{stack sorting}
(see~\cite[p.~14]{west90}~and~\cite{knuth73}).

A \df{left reflector} gadget is simply the image of a right reflector under
the reflection in the vertical line through the center of the permutation.
Thus it has swaps at locations
$$(n-i-j,j-i)$$
for $i$, $j$ and $b(\cdot)$ as before.  See \cref{f:left}.  The next result
follows immediately by symmetry.

\begin{lemma}\label{l-reflector}
A permutation can be performed by some left reflector if and only if it is
$213$-avoiding. This correspondence is bijective.
\end{lemma}

In our applications, we will prove and use two properties of $132$-avoiding (or
$213$-avoiding) permutations that are interesting in their own right: (i) any
permutation can be decomposed into a cyclic permutation and a $132$-avoiding
permutation (\cref{s:fish}); (ii) a $132$-avoiding permutation can be found
that maps any given subset of $\{1,\ldots,n\}$ to any other subset of the
same size (\cref{s:8n}).

\section{Logarithmic moves per path}\label{s:log}

Our simplest construction uses only splitters to obtain a simple tangle with
logarithmically many moves per path.

\begin{figure}
\centering
\begin{tikzpicture}[scale=.45]
\filldraw[fill=cyan!30] (0,0) -- (-6,6) .. controls (-2,8) and (2,4) .. (6,6) --cycle;
\filldraw[fill=black!20] (-6,-2) rectangle (-.2,-6);
\filldraw[fill=black!20] (6,-2) rectangle (.2,-6);
\path (-6,7.5)node{$1$} -- (6,7.5)node{$n$} node[midway]{$\cdots$};
\path (-6,-1)node{$\rho(1)$} -- (-1,-1)node{$\rho(m)$} node[midway]{$\cdots$};
\path (3.4,-1)node{$\rho(m+1)\cdots\rho(n)$};
\path (-6,-7)node{$\pi(1)$} -- (-1,-7)node{$\pi(m)$} node[midway]{$\cdots$};
\path (3.5,-7)node{$\pi(m+1)\cdots\pi(n)$};

\draw [thick,blue,decorate,decoration={brace,amplitude=10pt},yshift=16pt]
(-6,-1) -- (-.2,-1) node [above=10pt,blue,midway]{$L$};
\draw [thick,blue,decorate,decoration={brace,amplitude=10pt},yshift=16pt]
(.2,-1) -- (6,-1) node [above=10pt,blue,midway]{$R$};

\end{tikzpicture}
\caption{Construction for \cref{log}. After the initial splitter, the two rectangles signify
smaller recursively defined versions of the same construction.}
\label{f:log-construction}
\end{figure}
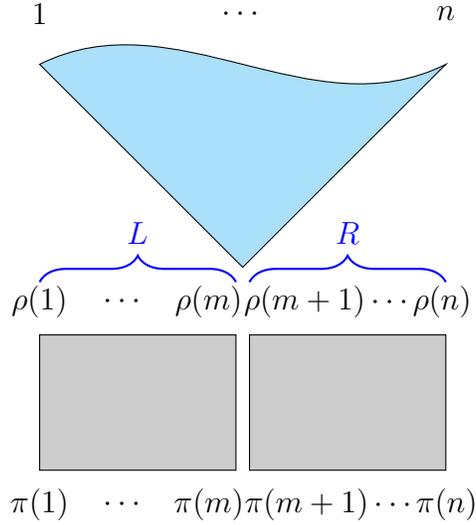
\begin{proof}[Proof of \cref{log}]
See \cref{f:log-construction} for the construction and \cref{f:log} for an
example. Let $\pi\in S_n$ be any permutation and let $m=\lfloor n/2\rfloor$.
Let $L=\{\pi(1),\ldots,\pi(m)\}$ and $R=\{\pi(m+1),\ldots,\pi(n)\}$. Consider
the Grassmannian permutation $\rho$ obtained by writing the elements of $L$
in increasing order followed by the elements of $R$ in increasing order.  By
\cref{splitter} there is a splitter than performs $\rho$.  We first apply
this splitter.  It remains to perform $\rho^{-1}\cdot\pi$, which is an
$(m,n-m)$-split permutation. Thus, we can split into two subproblems.  We
then recursively apply the same procedure to each, and place the resulting
tangles below the initial splitter, after appropriate translations.

Each path performs at most one move within each splitter
that it encounters (perhaps fewer, since some may splitters
involve no move for the path, and some pairs of splitters
may be positioned to abut one another,
so that two moves coalesce).  A path encounters at most
$\lceil \log_2 n\rceil$ splitters.

The tangle is simple, since if two paths cross in the first splitter, then
they subsequently remain in the two distinct halves.
\end{proof}

We remark that the above construction can be modified to
obtain a tangle with only one cluster, and $O(\log n)$
moves per path, thus matching up to constants the extremal case
$\theta\nearrow\tfrac12$ of \cref{counting}.  After the
first splitter, route path $\pi(m)$ alongside the
South-West boundary of the splitter to its correct position
$m$.  This path then remains vertical for the rest of the
tangle, keeping the two halves apart and preventing
formation of holes. Iterate on the two intervals $[1,m-1]$
and $[m+1,n]$, and ensure that the subsequent splitters are
translated upward until they touch some swap of a previous
stage.

\section{Bounded moves per path}\label{s:fish}

In this section we prove \cref{fish}.  The construction
will make essential use of reflector gadgets.  We use the
following key property of $312$-avoiding permutations, which
we will then extend to other patterns of length $3$. A
permutation is called \df{cyclic} if it has only one cycle
(or orbit).

\begin{lemma}
\label{cyclic} For any permutation $\pi \in S_n$, there exists a $312$-avoiding
permutation $\sigma$ such that $\sigma\cdot\pi$ is cyclic.
%It can be computed in $O(n)$ time.
\end{lemma}

\begin{proof}
Assume $n\geq 2$, otherwise the result is trivial. We use an iterative
procedure to compute a suitable $\sigma$. We start with $\pi$, and
pre-compose it by a sequence of suitably chosen disjoint cycles.  The
composition of these cycles will be $312$-avoiding.  Given the current
permutation $\tau$ (which is initially equal to $\pi$), a \df{rainbow
interval} is an interval $[a,b]$ such that all elements $i\in[a,b]$ belong to
distinct cycles of $\tau$.  A \df{maximal} rainbow interval $[a,b]$ is one
that is not a proper subset of another; thus, either we have $a=1$, or $a-1$
belongs to the same cycle as some element of $[a,b]$; a similar condition
holds at the other end.  If $\tau$ is not cyclic, then there exists some
maximal rainbow interval $[a,b]$ of length at least $2$.  We now replace
$\tau$ with the permutation $\tau':=\kappa\cdot\tau$, where
$$\kappa:=\Bigl[1,\ldots,a-1,\;\;\underbrace{a+1,a+2,\ldots,b,a},\;\; b+1,\ldots,n\Bigr],$$
(i.e.\ a rotation of the interval $[a,b]$; note that $\kappa$ is
$312$-avoiding). The effect of this change is to unite all the distinct
cycles of the elements of $[a,b]$ into one cycle; all other cycles are
unchanged. Consequently, if we iterate this operation, the rainbow intervals
used at successive steps will be disjoint, and eventually $\tau$ will be
cyclic. Moreover, the various cycles $\kappa$ used at different steps commute
with each other, and their composition $\sigma$ is $312$-avoiding.
\end{proof}

\begin{cor}\sloppypar
\label{more-cyclic} For any permutation $\pi \in S_n$, and any pattern
$p\in\{312,231,213,132\}$, there exists a $p$-avoiding permutation $\sigma$
such that $\sigma\cdot\pi$ is cyclic.
% It can be computed in $O(n)$ time.
\end{cor}

\begin{proof}
\cref{cyclic} is the case $p=312$.  Let $\rev:=[n,n-1,\ldots,1]\in S_n$ be
the reverse permutation.  For the case $p=231$, apply \cref{cyclic} to the
conjugate permutation $\rev\cdot\pi\cdot\rev^{-1}$ to obtain a $312$-avoiding
$\sigma$ with $\sigma\cdot\rev\cdot\pi\cdot\rev^{-1}$ cyclic.  The conjugate
of the last permutation by $\rev^{-1}$ is $\rev^{-1}\cdot
\sigma\cdot\rev\cdot\pi\cdot\rev^{-1}\cdot\rev = (\rev^{-1}\cdot
\sigma\cdot\rev)\cdot \pi$, which thus is cyclic also. The permutation
$\rev^{-1}\cdot \sigma\cdot\rev$ is $231$-avoiding, as required.

For $p=213$, apply \cref{cyclic} to $\rev\cdot \pi$, to
obtain a $312$-avoiding $\sigma$ with
$\sigma\cdot\rev\cdot\pi$ cyclic.  Then $\sigma\cdot\rev$
is $213$-avoiding.   Finally, for $p=132$, apply the
conjugation trick to the $p=213$ case.
\end{proof}

Here is the main step in the proof of \cref{fish}.
\begin{lemma}
\label{lm:2squares} For $n$ even and any $(n/2,n/2)$-split permutation $\pi
\in S_n$, there is a tangle with at most $4$ moves per path that performs
$\pi$.
\end{lemma}

\begin{figure}
\centering
\begin{subfigure}[c]{.4\textwidth}
\centering
\begin{tikzpicture}[scale=.3]
\filldraw[fill=red!20] (0,0)--(6,6)--(0,12)--(-6,6)--cycle;
\filldraw[fill=orange!20] (0,0)--(6,-6)--(0,-12)--(-6,-6)--cycle;
\filldraw[fill=green!30!cyan!30] (0,0)--(6,6)--(6,-6)--cycle;
\filldraw[fill=green!70!yellow!30] (0,0)--(-6,6)--(-6,-6)--cycle;

\draw[thick,->] (-2,8) to[bend left=30]node[left,midway]{$\alpha$} (-2,4) ;
\draw[thick,->] (2,8) to[bend right=30]node[right,midway]{$\alpha^{-1}_{\phantom{-1}}$} (2,4) ;
\draw[thick,->] (-4,2) to[bend right=30]node[right,midway]{$\sigma_1$} (-4,-2) ;
\draw[thick,->] (4,2) to[bend left=30]node[left,midway]{$\sigma_2$} (4,-2) ;
\draw[thick,->] (-2,-4) to[bend left=30]node[left,midway]{$\beta$} (-2,-8) ;
\draw[thick,->] (2,-4) to[bend right=30]node[right,midway]{$\beta^{-1}_{\phantom{-1}}$} (2,-8) ;
\end{tikzpicture}
\caption{The construction for \cref{lm:2squares}: two matrix gadgets, a left
reflector and a right reflector.}
\end{subfigure}
\hspace{.1\textwidth}
\begin{subfigure}[c]{.4\textwidth}
\centering
\includegraphics[width=.8\textwidth]{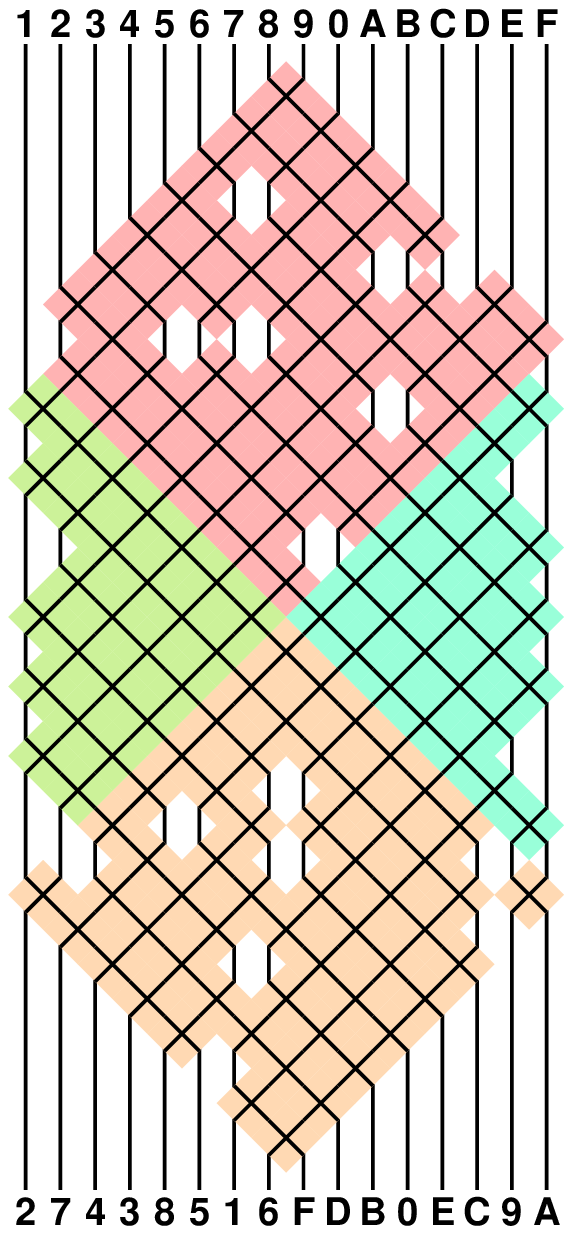}
\caption{An example.}
\end{subfigure}
\caption{}\label{f:fish-body}
\end{figure}
\begin{proof}[Proof of \cref{lm:2squares}]
 Let $n=2m$.  Since $\pi$ is $(m,m)$-split, there exist $\pi_1,\pi_2\in S_m$ such that
$\pi_1=(\pi(1),\ldots,\pi(m))$ and
$\pi_2=(\pi(m+1)-m,\ldots,\pi(2m)-m)$. We construct the
required tangle using
two matrix gadgets, one above the other, %(Lemma~\ref{lm:square})
together with a left reflector and a right reflector (each of width $m$) in
the two spaces between them, as in \cref{f:fish-body}.

Let $\alpha,\beta\in S_m$ be the permutations indexing the upper and lower
matrix gadgets respectively (see the definition of a matrix gadget).  Let
$\rho_1,\rho_2\in S_m$ be the permutations performed by the left reflector
and the right reflector respectively.  Clearly such a tangle performs an
$(m,m)$-split permutation, for any choices of $\alpha,\beta,\rho_1,\rho_2$.
Our task is to choose these permutations so as to perform the required $\pi$.

Recall from \cref{lm:square} that a matrix gadget performs its indexing
permutation on the left and the inverse permutation on the right.  Thus, our
tangle performs $\pi$ if and only if
\begin{equation} \label{eq1}
\alpha\cdot\rho_1\cdot\beta=\pi_1 \text{\quad and\quad}
\alpha^{-1}\cdot\rho_2\cdot\beta^{-1}=\pi_2.
\end{equation}
The first equation gives $\pi_1^{-1}\cdot \alpha\cdot\rho_1=\beta^{-1}$, and
substituting into the second gives
$\alpha^{-1}\cdot\rho_2\cdot\pi_1^{-1}\cdot \alpha\cdot\rho_1=\pi_2$.
Rearranging,
\begin{equation} \label{eq2}
\rho_2\cdot\pi_1^{-1}=\alpha\cdot (\pi_2\cdot\rho_1^{-1})\cdot\alpha^{-1}.
\end{equation}
There exists an $\alpha$ satisfying \eqref{eq2} if and only if the two
permutations $\rho_2\cdot\pi_1^{-1}$ and $\pi_2\cdot\rho_1^{-1}$ are
conjugate.  By \cref{more-cyclic}, for any $\pi_1$, we can choose a
$132$-avoiding $\rho_2$ such that $\rho_2\cdot\pi_1^{-1}$ is cyclic.
Similarly, for any $\pi_2$, we can choose a $213$-avoiding $\rho_1$ such that
$\rho_1\cdot\pi_2^{-1}$ is cyclic, whence the inverse $\pi_2\cdot\rho_1^{-1}$
is cyclic also.  The permutations $\rho_1,\rho_2$ can be performed by the
appropriate reflector gadgets by \cref{r-reflector,l-reflector}.  Thus, the
two permutations mentioned above are both cyclic, and therefore conjugate,
and so we can choose $\alpha$ satisfying \eqref{eq2}.  Finally, we can
compute $\beta=\rho_1^{-1}\cdot\alpha^{-1}\cdot\pi_1$, and \eqref{eq1} will
be satisfied.

The resulting tangle has at most $4$ moves per path: a path has two moves in
each matrix gadget, and these moves continue into the reflectors, since the
gadgets abut each other.
\end{proof}

\begin{proof}[Proof of \cref{fish}]
First consider even $n=2m$.  See \cref{f:fish} for an example.  Using
\cref{splitter}, we first apply a splitter gadget that performs the
permutation $\tau$, where $\tau(1),\ldots,\tau(m)$ are $\pi(1),\ldots,\pi(m)$
in increasing order, and $\tau(m+1),\ldots,\tau(2m)$ are
$\pi(m+1),\ldots,\pi(2m)$ in increasing order.  We then use
\cref{lm:2squares} to obtain a tangle that performs the $(m,m)$-split
permutation $\tau^{-1}\cdot\pi$, and we place this tangle below the splitter.
The splitter adds at most one move to each path.

For odd $n=2m+1$, we modify the construction as shown in \cref{odd-fish}. The
initial splitter separates the paths into sets of sizes $m$ and $m+1$, with
path at the extreme right being $z=\max\{\pi(m+1),\ldots,\pi(2m+1)\}$.  We
then proceed as before for the first $2m$ paths.  Finally, we insert path $z$
into its proper place in $\pi$ by an L-move alongside the South-East side of
the lower matrix gadget. Path $z$ has only $2$ moves, and every other path
still has at most $5$ moves.
\end{proof}
\begin{figure}
\centering
\includegraphics[width=.24\textwidth]{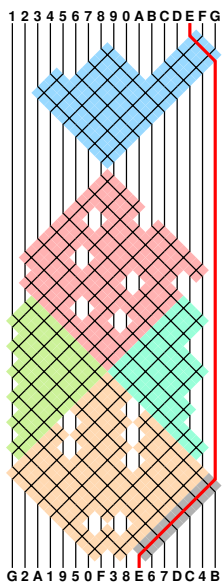}
\caption{The construction for \cref{fish} for odd $n$: the largest element
in the right half of the permutation is routed along the right side.}
\label{odd-fish}
\end{figure}

We remark that the last trick for adding an additional path with only $2$ moves could be iterated, to obtain an inductive construction of larger tangles.
However, in general this would incur a quadratic number of moves in total, for
similar reasons to the construction in \cref{bubble}.

\section{Linear total moves}\label{s:8n}

We begin with another fact about $132$-avoiding permutations.  Write
$[k]:=\{1,\ldots, k\}$.

\begin{lemma}
\label{lm:mixer} If $A,B\subseteq [n]$ have equal cardinality then there
exists a $132$-avoiding $\pi\in S_n$ with $\pi(A)=B$.
\end{lemma}

\begin{proof}
The proof is by induction on $n$. If $n = 1$ then the claim is obvious.
Suppose that the theorem holds for all $n'<n$.  We will deduce it for $n$. A
pair $(i,j)$ is called \df{conforming} if either $i\in A$ and  $j\in B$, or
$i\notin A$ and $j\notin B$.  (In other words, if we are allowed to assign
$\pi(i)=j$).  We consider several cases.

\paragraph{Case 1.} Pair $(n,1)$ is conforming.  Without loss of generality,
suppose that $n\notin A$ and $1\notin B$; otherwise take complements of $A$
and $B$. Consider the set $B-1:=\{i -1 : i \in B\}$.  By the induction
hypothesis, there exists a $132$-avoiding $\sigma \in S_{n-1}$ with
$\sigma(A)=B-1$. Define $\pi\in S_n$ by setting $\pi(n)=1$, and
$\pi(i)=\sigma(i)+1$ for $i<n$.  Then $\pi$ is 132-avoiding, and maps $A$ to
$B$, as required.

\paragraph{Case 2.} Pair $(1,n)$ is conforming. Without loss of generality,
$1 \notin A$ and $n \notin B$. Consider $A-1:=\{i -1 : i \in A\}$. By the
induction hypothesis there exists a $132$-avoiding $\sigma \in S_{n-1}$ with
$\sigma(A-1)=B$. Define $\pi\in S_n$ by $\pi(1)=n$ and $\pi(i)=\sigma(i - 1)$
for $i > 1$.

\sloppypar \paragraph{Case 3.} Pair $(n,n)$ is conforming.  Apply the
inductive hypothesis to $1,\ldots,n-1$ and set $\pi(n) = n$.

\paragraph{Case 4.} None of the pairs $(n,1)$, $(1,n)$, $(n,n)$ is conforming.
Without loss of generality, $1 \in A$. Then $n\notin B$ because $(1,n)$ is
not conforming. Then $n\in A$ because $(n,n)$ is not conforming. Then
$1\notin B$ because $(n,1)$ is not conforming.  In summary, we have $1,n\in
A$ but $1,n\notin B$.

\medskip
We claim that there exists an integer $k$ with $2\leq k\leq n-2$ such that
$|A \cap [k]| = |B \cap ( [n] \setminus [n - k])|$.  Indeed, we have $|A \cap
[1]| = 1 > 0 = |B \cap ( [n] \setminus [n - 1])|$, whereas $|A \cap [n-1]| =
|A|-1 < |B| = |B \cap ( [n] \setminus [1])|$; but the difference $|A \cap
[j]| - |B \cap ( [n] \setminus [n - j])|$ decreases by at most $1$ as $j$ is
increased by $1$; thus it must be $0$ for some $j$.

Let $A'=A \cap [k]$ and $B'=(B\cap ( [n] \setminus [n-k]))-(n-k)$. By the
induction hypothesis, (since $k<n$) there exists a $132$-avoiding $\pi_1 \in
S_{k}$ with $\pi_1(A')=B'$.   Let $A''=(A \cap ( [n] \setminus [k])) - i$ and
$B''=B\cap [n-k]$. By the induction hypothesis, (since $n-k<n$) there exists
a $132$-avoiding $\pi_2 \in S_{n-k}$ with $\pi_2(A'')=B''$.  We define $\pi$ by
setting $\pi(j) = \pi_1(j)+n-k$ for $j\le k$, and $\pi(j) = \pi_2(j-k)$ for
$j > k$. This $\pi$ is $132$-avoiding: if $u<v<w$ form a $132$ pattern, then we
cannot have all three in $[k]$ or all three in $[n]\setminus [k]$.  On the
other hand, we cannot have $u\leq k < w$: indeed, for all $i\leq k<j$ we have
$\pi(i)>n-k\geq \pi(j)$.
\end{proof}

The following is a major ingredient of the proof of \cref{8n}.

\begin{prop}\label{c-gadget}
Let $\pi\in S_n$ be a $(\lceil n/2\rceil,\lfloor n/2\rfloor)$-split
permutation.  The permutation $\pi$ can be performed by a tangle all of whose
swaps are within the triangular region $\{(x,t): -x< t < x\}$.  The tangle
accepts $n$ paths running in the South-East direction on its North-West edge,
and outputs them running in the South-East direction on its South-East edge,
and has at most $4n$ moves including these input and output segments.
\end{prop}

\begin{figure}
\centering
\begin{subfigure}[t]{.4\textwidth}
\centering
\begin{tikzpicture}[scale=.25]
\filldraw[fill=red!20] (0,0)--(6,6)--(0,12)--(-6,6)--cycle;
\filldraw[fill=green!80!black!20] (0,12)--(6,18)--(6,6)--cycle;
\filldraw[fill=black!30] (0,0)--(6,6)--(6,-6)--cycle;
%\shadedraw[left color=red!50, right color=green!30] (0,0)--(6,6)--(6,-6)--cycle;
\foreach \x in {.5,...,11.5}
 \draw (-6+\x,6+\x)--(-7+\x,7+\x);
\foreach \x in {.5,...,11.5}
 \draw (-6+\x,6-\x)--(-7+\x,5-\x);
\node at (0,6){$M$};
\node at (3,12){$R$};
\node at (3,0){$C'$};
\draw[very thick,red] (-5.5,8.5)--(-1.5,4.5)--(-1.5,3.5)--(-3.5,1.5);
\draw[very thick,green!70!black] (2.5,16.5)--(5.5,13.5)--(5.5,8.5)--(2.5,5.5)--(2.5,4.5)--(3.5,3.5);
\end{tikzpicture}
\caption{Construction of the tangle $C$ in \cref{c-gadget},
comprising a right reflector, a matrix, and a recursively-constructed version $C'$
of itself.}\label{f:c-gadget}
\end{subfigure}
\hfill
\begin{subfigure}[t]{.25\textwidth}
\centering
\includegraphics[width=\textwidth]{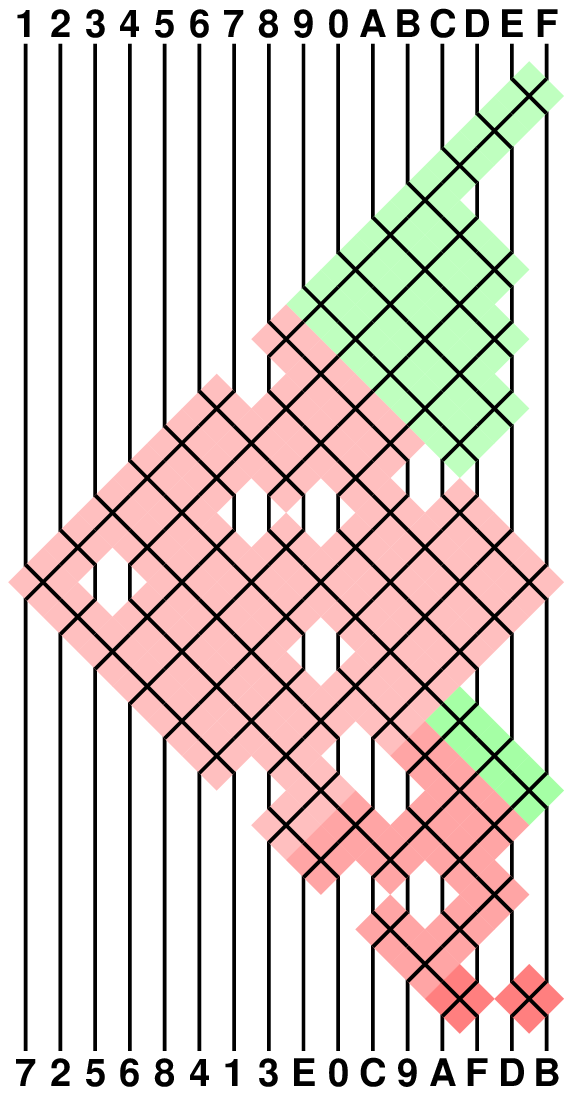}
\caption{An example for even $n$.}\label{f:c-gadget-even}
\end{subfigure}
\hfill
\begin{subfigure}[t]{.25\textwidth}
\centering
\includegraphics[width=\textwidth]{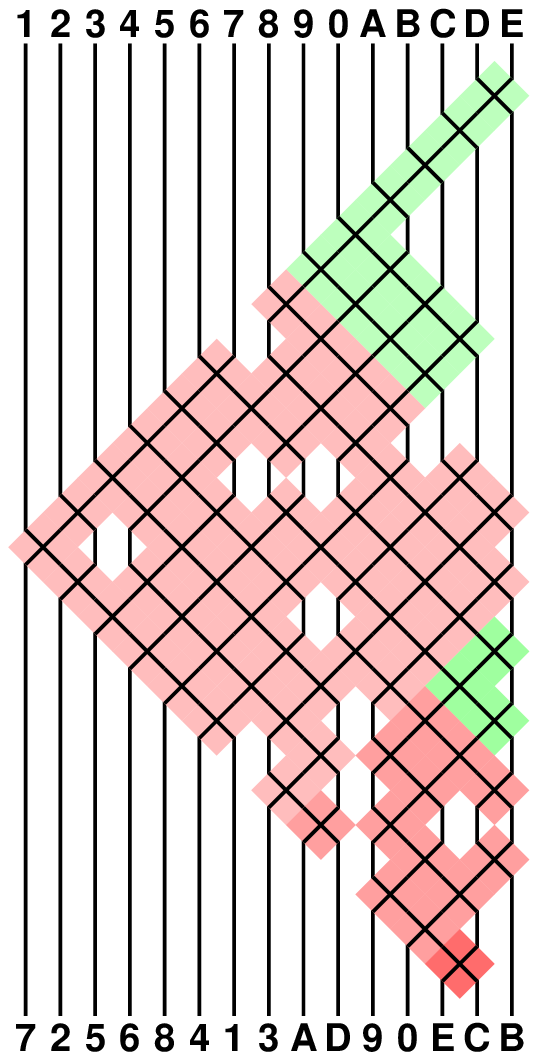}
\caption{An example for odd $n$.}\label{f:c-gadget-odd}
\end{subfigure}
\caption{}
\end{figure}
\begin{proof}
We first assume that $n=2m$ is even, so $\pi$ is $(m,m)$-split.  The
construction of the required tangle $C$ is recursive: it consists of a matrix
gadget $M$, together with a right reflector $R$ of width $m$ placed to the North-East
of the matrix, and a smaller, recursively-constructed version $C'$ of itself
(performing a suitable permutation of size $m$) placed to the South-East of
the matrix.  See \cref{f:c-gadget}.

We now explain how to choose the gadgets.  Let $\rho,\mu\in S_{n}$ be the
permutations performed by the right reflector $R$ (when translated to the
right half $[m+1,n]$) and the matrix gadget $M$, respectively.  Since the
right-reflector does not affect positions in the left half $[1,m]$, we
require that $[\mu(1),\ldots,\mu(m)]=[\pi(1),\ldots,\pi(m)](\in S_m)$.
Therefore we choose the matrix gadget to be indexed by this last permutation.
Now consider the right half. The tangle $C'$ can perform any desired $(\lceil
m/2\rceil,\lfloor m/2\rfloor)$-split permutation on positions
$m+1,\ldots,2m$. Therefore, letting $Q=[m+\lceil m/2\rceil+1,n]$ be the set
of positions in the last quarter of $[1,n]$, we need to choose $\rho$ so that
$\rho\cdot\mu(Q)=\pi(Q)$.  Since $|\mu(Q)|=|\pi(Q)|$, by
\cref{lm:mixer,r-reflector}, there is a right reflector that achieves this.

In the case when $n=2m+1$ is odd, the construction is
modified as follows. The matrix gadget is replaced with a
truncated version (with the rightmost swap deleted), so
that we may choose it to perform the required permutation
on positions $1,\ldots,m+1$.

Finally, we count moves.  Suppose that all paths start running in the
South-East direction.  Then each path makes at most 2 moves in the reflector
together with the matrix, including the input path, but not including the
final R-move in the case of the paths in the right half.  Since these moves
continue into $C'$, writing $A(n)$ for the maximum number of moves required
by our construction for a permutation of size $n$, we have
$$A(n)\leq 2n+A\bigl(\lfloor n/2 \rfloor\bigr).$$
By induction, $A(n)\leq 4n$.
\end{proof}

\begin{figure}
\centering
\begin{tikzpicture}[scale=.5]
\filldraw[fill=yellow!40] (0,0) -- (3,3) -- (0,6) -- (-3,3) --cycle;
\filldraw[fill=orange!40] (-6,0) -- (-3,-3) -- (0,0) -- (-3,3) --cycle;
\filldraw[fill=blue!30] (0,0) -- (-6,-6) .. controls (-2,-4) and (2,-8) .. (6,-6) --cycle;
\filldraw[fill=cyan!30] (3,3) -- (0,6) .. controls (2,7) and (4,5) .. (6,6) --cycle;
\filldraw[fill=cyan!30] (-3,3) -- (-6,6) .. controls (-4,7) and (-2,5) .. (0,6) --cycle;
\filldraw[fill=black!20] (0,0)--(6,6)--(6,-6)--cycle;
\node at (-3,0) {$M_2$};
\node at (0,3) {$M_1$};
\node at (3.5,0) {$C$};
\node at (-3,4.75) {$S_1$};
\node at (3,4.75) {$S_2$};
\node at (0,-3.5) {$G$};
\foreach \x in {-6,-3,0,3}
\draw[thick,<->] (\x+.1,7)--(\x+2.9,7) node[midway,above]{$m$};
\end{tikzpicture}
\caption{Construction for \cref{8n}: splitters $S_1,S_2$, matrix gadgets $M_1,M_2$, merger $G$,
and a tangle $C$ from \cref{c-gadget}.}
\label{f:8n-construction}
\end{figure}
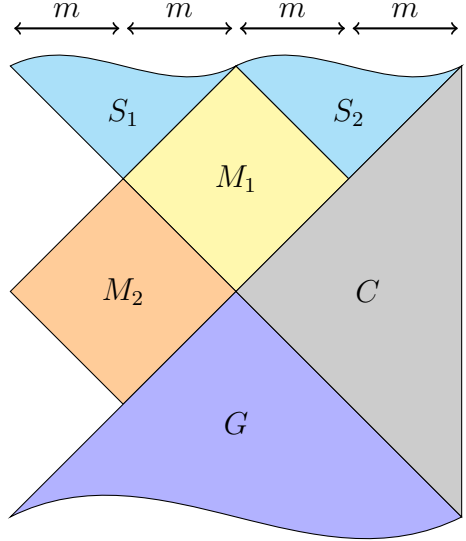
\begin{proof}[Proof of \cref{8n}]
The construction is illustrated in \cref{f:8n-construction}, and \cref{f:8n}
is an example. We first assume that $n$ is a multiple of $4$, and write
$n=4m$. As shown in \cref{f:8n-construction}, the tangle finishes with a
merger $G$ that (by \cref{merger}) intersperses the paths in locations
$1,\ldots, 2m$ with those in $2m+1,\ldots, 4m$ in an arbitrary way while
maintaining the relative order of each.  Therefore, the remainder of the
tangle (above the merger) needs to perform an arbitrary $(2m,2m)$-split
permutation.  On the other hand, the tangle starts with two splitters $S_1$
and $S_2$, placed in the first and second halves.  By \cref{splitter} each of
these splitters can map any desired set of paths into its own first half (of
width $m$). Therefore, the task for the remaining portion of the tangle
(i.e.\ everything apart from the merger and the two splitters) is to perform
an arbitrary $(m,m,m,m)$-split permutation.

The remainder of the tangle is composed of two matrix gadgets, together with
a tangle constructed via \cref{c-gadget}.  Both matrix gadgets have width
$2m$.  The upper matrix gadget $M_1$ occupies the middle half $[m+1,3m]$ of
$[1,n]$. The other matrix gadget, $M_2$, abuts $M_1$ to the South-West and
occupies the first half $[1,2m]$.  The tangle $C$ from \cref{c-gadget} also
has width $2m$, and is located on the right, partially abutting $M_1$.

We now explain how to choose these gadgets.  The matrix gadget $M_2$ is
chosen so as to perform the required permutation in the first quarter
$[1,m]$.
Then $M_1$ chosen so that the required permutation in the second quarter
$[m+1,2m]$ is performed by the left half of $M_1$ composed with the right half of $M_2$.
 Finally, $C$ needs to perform an arbitrary $(m,m)$-split permutation (on positions
$[2m+1,4m]$).  This can be achieved, by \cref{c-gadget}.

We now count moves.  We first total the moves within each component.  When
two components abut each along a common boundary, the moves crossing this
boundary will be double-counted.  Therefore we then subtract a term
corresponding to the total length of the common boundaries.  The upper
splitters each contribute $2m$ moves; the two matrix gadgets each contribute
$4m$ moves; the final merger contributes $4m$ moves; and the tangle $C$
contributes $4(n/2)=2n$ moves, by \cref{c-gadget}.  The total over-counting
from common boundaries is $m+m+3m+3m$.  Therefore, there are at most
$24m-8m=16m=4n$ moves.

\begin{figure}
\begin{subfigure}[b]{.24\textwidth}\centering
\includegraphics[width=\textwidth]{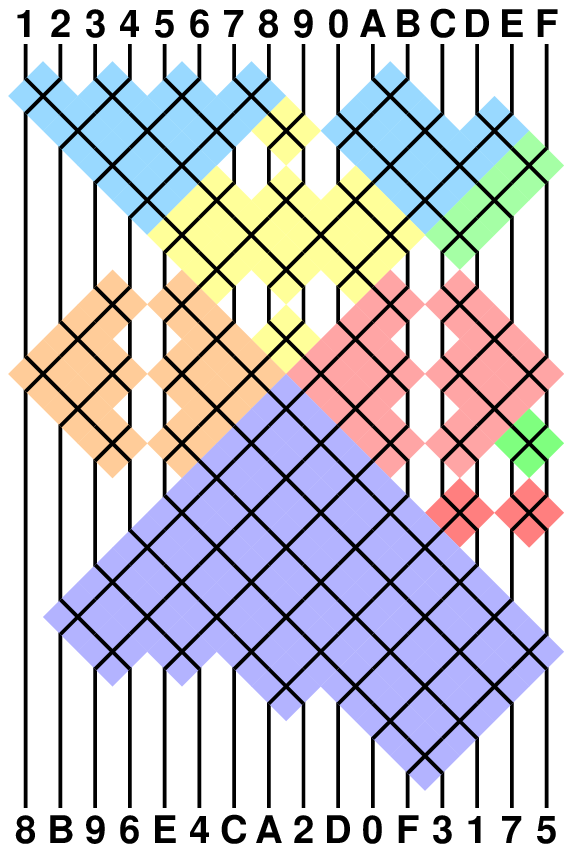}
\caption{$n\equiv 0\mod 4$.}
\end{subfigure}
\hfill
\begin{subfigure}[b]{.24\textwidth}\centering
\includegraphics[width=\textwidth]{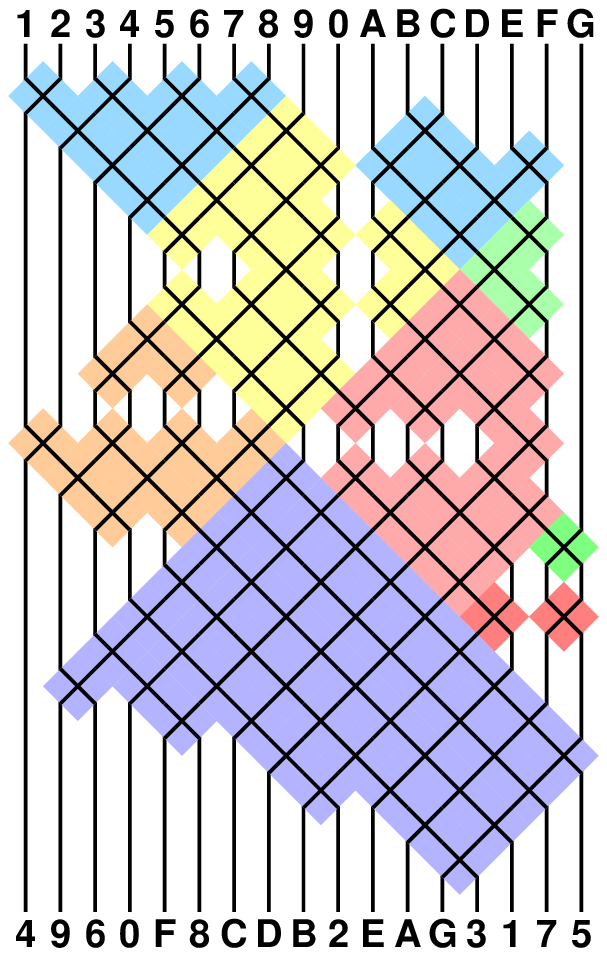}
\caption{$n\equiv 1\mod 4$.}
\end{subfigure}
\hfill
\begin{subfigure}[b]{.24\textwidth}\centering
\includegraphics[width=\textwidth]{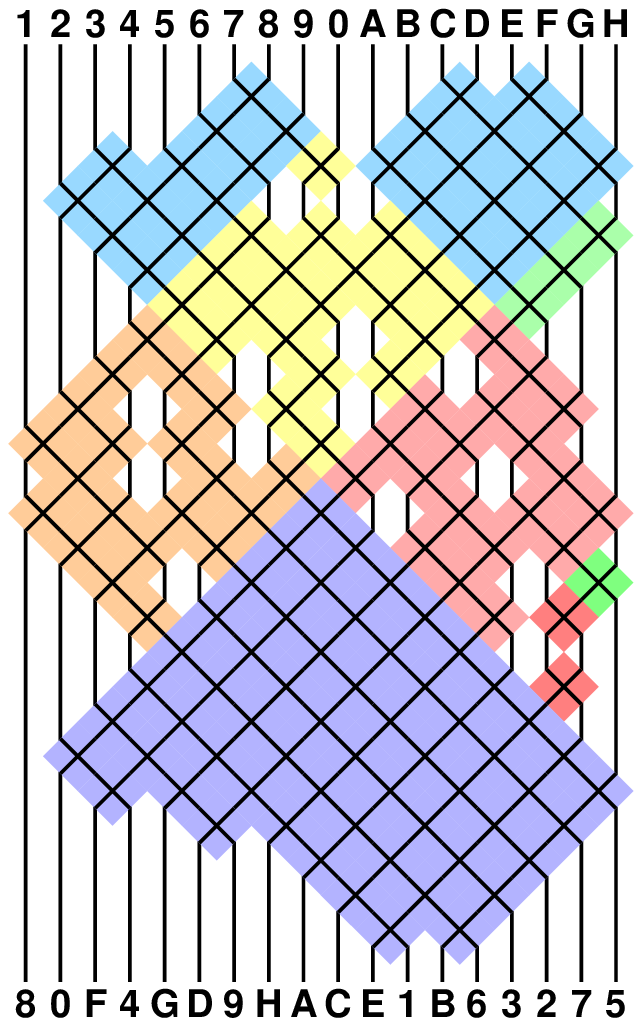}
\caption{$n\equiv 2\mod 4$.}
\end{subfigure}
\hfill
\begin{subfigure}[b]{.24\textwidth}\centering
\includegraphics[width=\textwidth]{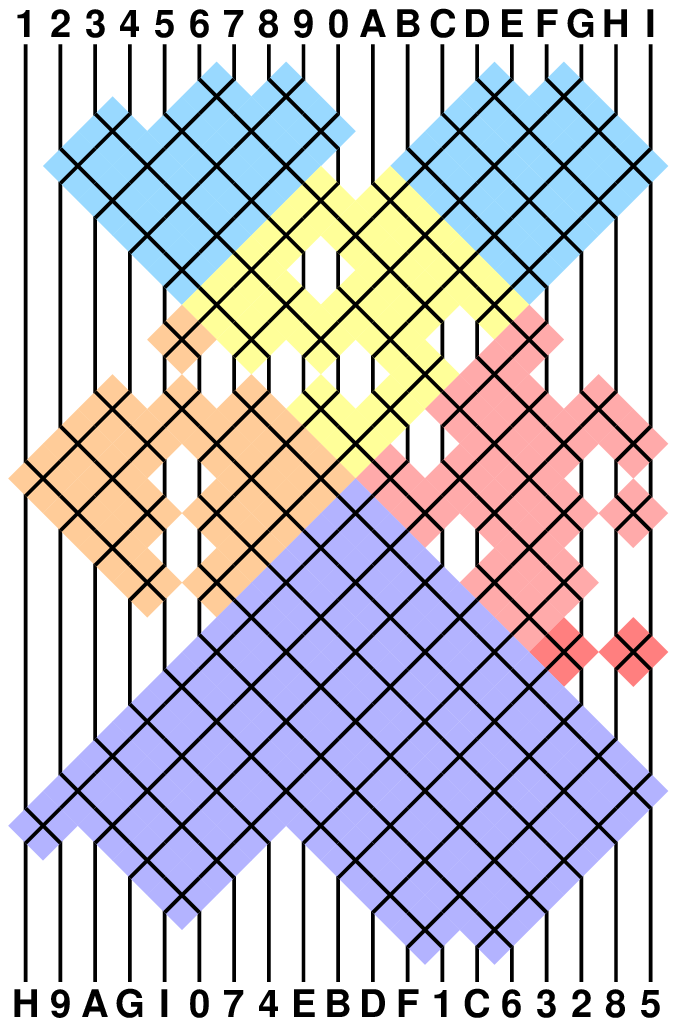}
\caption{$n\equiv 3\mod 4$.}
\end{subfigure}
\caption{Variations of the construction for \cref{8n},
according to the congruence class of $n$.}
\label{mod4}
\end{figure}

Finally, we describe how the construction is adjusted when $n$ is not a
multiple of $4$.  Let $n=4m+r$ where $m$ is an integer and $r\in\{0,1,2,3\}$.
Depending on the value of $r$, we choose a suitable splitting into quarters,
and use carefully chosen truncated matrix gadgets.  The splitters and merger
are adjusted to that the remaining central section of the tangle must perform
a permutation that is split as follows:
\begin{equation*}
\begin{array}{rllll}
r=0:&\quad (m,&m,&m,&m)\\
r=1:&\quad (m,&m,&m+1,&m)\\
r=2:&\quad (m,&m+1,&m+1,&m)\\
r=3:&\quad (m+1,&m+1,&m+1,&m).
\end{array}
\end{equation*}
The case $r=0$ was described above.  In the case $r=1$, the matrix gadget
$M_1$ is not truncated, but has width $2(m+1)$, and is chosen to have no swap
in its leftmost column.  In the case $r=2$, the matrix $M_2$ is truncated on
its left side.  In the case $r=3$, both matrices have width $2(m+1)$, and
neither is truncated. For each of $r=1,2,3$, the tangle $C$ has odd width,
and performs a $(m+1,m)$-split permutation, as stated in \cref{c-gadget}.
These choices ensure that the various components can still abut each other
without introducing extra moves at the boundaries.  See \cref{mod4} for examples.
\end{proof}

We remark that, in the above construction, while the average number of moves
per path is only $4$, some paths may have as many as $\Theta(\log n)$ moves
-- this is a consequence of the recursive construction in \cref{c-gadget}.

\section{Cluster bound}
\label{s:clusters}

In this section we prove \cref{counting}.  Recall that swaps are located at
elements of the integer lattice $\ZZ^2$, and thus corners are located at
elements of $(\ZZ+\tfrac 12)^2$.  Recall that a cluster is a connected
component of the graph whose vertices are corners, and with an edge between
two corners if their locations are within $\ell^\infty$-distance $1$.  See
\cref{f:clusters} for an example.
\begin{figure}
\centering
\includegraphics[width=.4\textwidth]{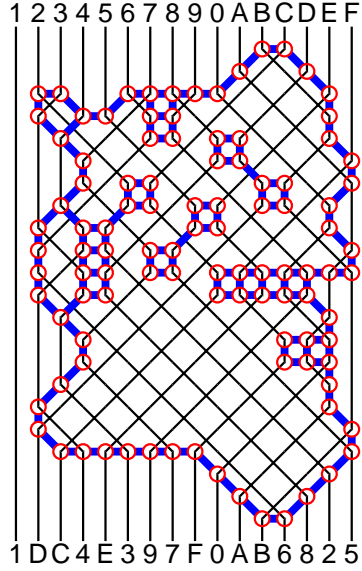}
\caption{Corners and clusters (for a tangle constructed according to the proof of \cref{8n}).
Corners are circled, and corners connected by thick lines belong to the same cluster.
 There are three clusters.}
\label{f:clusters}
\end{figure}

We start with a standard estimate for counting clusters. Let $\ZZ^2_*$ be the
graph with vertex set $\ZZ^2$ and an edge between any two elements that are
at $\ell^\infty$-distance $1$ from each other.  By a \df{$*$-animal} we mean
a finite subset of $\ZZ^2$ that induces a connected subgraph of $\ZZ^2_*$.
The \df{size} of a $*$-animal is the number of its vertices. Two $*$-animals
are said to be \df{equivalent} if one can be obtained from the other by a
translation of $\ZZ^2$.

\begin{lemma}\label{animals}
The number of equivalence classes of $*$-animals of size $m$ is at most
$A^m$, where $A=7^7/6^6$.
\end{lemma}

\begin{proof}
Apply the argument of Eden \cite{eden}, adapted to the $*$ lattice.   See
also e.g.~\cite{klarner}.
\end{proof}

\begin{proof}[Proof of Theorem~\ref{counting}]
Fix $\theta\in(0,\tfrac12)$ and $n>\theta^{-8/\theta}$.  Suppose for a
contradiction that there are at least $e^{-n}n!$ distinct permutations
$\pi\in S_n$ each of which has a tangle $T_\pi$ with fewer than
$K:=(\tfrac12-\theta)n$ clusters and fewer than $C:= \tfrac16 \theta n
\log n$ corners.

For any tangle $T$, suppose that there are no corners at
the time $t+\tfrac12$.  It is easy to check that all segments must be vertical
at the point, so the three permutations corresponding to
times $t-\tfrac12,t+\tfrac12,t+\tfrac32$ are all equal.
Therefore we can remove one of these permutations from the
sequence to obtain a new tangle. This operation preserves
the number of corners, and does not increase the number of
clusters.  We can therefore assume that each of the tangles
$T_\pi$ defined above has depth at most equal to its number
of corners.  We may further assume that the time of the
first corner is $\tfrac12$. Therefore all corners are
within a fixed rectangle $R$ of area $Cn$.  (Recall that
there are at most $C$ corners).

If we are given the set of locations of corners of a
tangle, together with the directions of the two incident
path segments at each corner, then we can recover the
tangle. At any given corner there are at most $3^2-3=6$
possible choices for this pair of directions.

We now bound from above the number of possible tangles $T_\pi$.  A cluster of
size $m$ corresponds to a $*$-animal together with a location in the
rectangle $R$.  Therefore the number of possible tangles is at most
$$\sum_{m_1,\ldots, m_k} \prod_{i=1}^k \bigl( A^{m_i}  6^{m_i} Cn\bigr),$$
where the sum is over all sequences $(m_i)_{i=1,\ldots k}$ with $k\leq K$,
and $m_i\geq 1$ and $\sum_i m_i\leq C$, and where $A$ is the constant from
Lemma~\ref{animals}. The number of choices of such $(m_i)_{i=1,\ldots ,k}$ is
at most $2^{C}$, so the above expression is at most $(Cn)^K (12A)^C$.

Since each $T_\pi$ corresponds to a different permutation $\pi$, we have
$$e^{-n} n! \leq (Cn)^K (12A)^C.$$
Taking logarithms, substituting for $C$ and $K$, and using $\log (n!)>n\log
n- n$, we obtain
$$n \log n - 2n \leq (\tfrac12 -\theta)n \log(Cn) +
\tfrac16 \theta n \log n \log(12 A).$$
Using $\log(12 A)<6$ and simplifying gives
$$ \tfrac12 \log n -2 \leq (\tfrac12 -\theta) \log C.$$
Since $0<\tfrac12-\theta<\tfrac12$ and $\log C\leq \log n +\log\log n$,
this implies
$$\theta \log n \leq 2+\tfrac12  \log\log n.$$
It is straightforward to check that this gives a contradiction if
$n>\theta^{-8/\theta}$.
\end{proof}

We remark that there is nothing special about the choice of
$\ell^\infty$-distance $1$ in the definition of clusters, except that it
is fairly natural in the context of the tangles that we have constructed.  The above
argument goes through (with different constants) for other choices of norm
and threshold distance.

\section{Lower bound for simple tangles}
\label{s:simple-lower}

\begin{proof}[Proof of \cref{lower}]
First assume that $n=r^2+2$. Consider the permutation
\begin{multline}
\pi=\\ \Bigl[ n,\;\; \underbrace{r+1, r, \dots, 2},\;\; \underbrace{2r+1, 2r,
\dots, r+2}, \;\;\dots,\;\; \underbrace{n-1, \dots, n-r},\;\;1\Bigr].
\end{multline}
 Thus, $\pi$ consists of $r$ blocks of length $r$, with each block
having its elements in reverse order, and with $1$ and $n$ in reverse order
at  the two ends. For example, for $r = 3$ the permutation is $\pi = [11,\;
4, 3, 2, \;7, 6, 5,\; 10, 9, 8, \;1]$.

Define \df{block} $i$ to be the set $B(i) = \{ir+2, \dots, ir+r+1\}$ for $i
=0, 1, \dots, r-1$. Let $T$ be a simple tangle that performs $\pi$. Every
path other than $1$ and $n$ has at least two moves, since it crosses paths
$n$ and $1$ in different directions.  Observe that paths of $B(i)$ and $B(j)$
do not cross each other for $i \neq j$. Call a path \df{bad} if it has at
least $3$ moves, and call a block \df{terrible} if it contains at most one
non-bad path. Next, we show that there are at most $3$ non-terrible blocks,
from which the result will follow easily.

Since paths from different blocks cannot cross each other, for any $i<j$, all
elements of $B(i)$ precede all elements of $B(j)$ in any permutation of the
tangle.  Now consider the location $(x,t)$ of the unique swap between paths
$1$ and $n$.  Recall that $(x,t)$ occurs between permutations $\pi_t$ and
$\pi_{t+1}$, and swaps the elements in locations $x$ and $x+1$. Let
$$H:=\bigl\{\pi_t(x-r-1),\ldots,\pi_t(x+r)\bigr\}$$
be the set of elements that are within distance $r$ on the left and right
just before this swap. The set $H$ has exactly $2r$ elements including $1$
and $n$.  Thus, by the previous observation, $H$ contains elements from at
most $3$ blocks. We will show that any block having no elements in $H$ is
terrible.

Suppose that $B(i)\cap H=\emptyset$.  By the argument of the previous
paragraph, either all elements of $B(i)$ are before all elements of $H$ in
the permutation $\pi_t$, or they are after. Without loss of generality,
assume the former. This implies that each path of $B(i)$ crosses $1$ before
it crosses $n$. Let $p<q$ be any elements of  $B(i)$.  We will show that at
least one of paths $p,q$ is bad.  Let $(y,s)$ be the location of the swap of
$p$ and $q$, and consider the permutation $\pi_{s}$. We consider six cases.

Suppose first that $\pi_{s}=[\ldots,1,\ldots,p,q, \ldots,n,\ldots]$ (which is
to say that $p$ and $q$ swap in the region above paths $1$ and $n$). Path $p$
has an R-move (to swap with $q$), then an L-move (to swap with $1$), then an
R-move (to swap with $n$). Therefore $p$ is bad. The case
$\pi_{s}=[\ldots,n,\ldots,p,q, \ldots,1,\ldots]$ (where $p$ and $q$ swap
below paths $1$ and $n$) can be treated symmetrically.

Suppose now that $\pi_{s}=[\ldots,p,q,\ldots,1,\ldots,n,\ldots]$ (which is to
say that $p$ and $q$ swap in the region left of paths $1$ and $n$, and at or
before time $t$, so $s\leq t$). The argument for this case is illustrated in
\cref{interval}. If $p$ is not bad, then path $p$ has an L-move (to swap with
1), followed by an R-move during which it swaps with both $q$ and $n$. Let
$H':=\{\pi_t(x-r-1),\ldots,\pi_t(x)\}$.  All elements of $H'$ are between $p$
and $n$ in $\pi_t$.  These elements do not swap with $p$ after time $t$,
because $\pi_t(x)=1$ has already swapped with $p$, while the others belong to
different blocks and so never swap with $p$.  Let $u$ be the time of the swap
of $p$ and $n$. Since $u\geq t$, all elements of $H'$ must swap with $n$
strictly before time $u$. Therefore $u-t\ge r$.  Therefore, path $p$ has at
least $r$ swaps at times in the interval $[t,u)$ (since $s\leq t$, so its
unique R-move is in progress throughout this interval). Since path $p$ also
swaps with $1$ and $n$,
 it has at least $r+2$ swaps in total,
which contradicts simplicity, since $p$ is involved in only $r+1$ inversions.
Thus, $p$ is bad. The case $\pi_{s}=[\ldots,p,q, \ldots,n,\ldots,1,\ldots]$
can be treated symmetrically.
\begin{figure}
\centering
\begin{tikzpicture}[thick,scale=0.35]
\filldraw[red,fill=red!30,rounded corners=5pt] (5.5,-9.5) rectangle (9.5,-8.5);
\node[red] at (5,-8){$H'$};
\draw[red,->,very thick] (6,-9.5) to[bend right] (6,-11.5);

\draw(1,-1) node[above left]{$1$}--(9,-9);
\draw(10,-9)node[above right]{$n$}--(4,-15);
\draw[blue](2,-1)node[above right]{$p$}--(-2,-5)--(-2,-8)--(5,-15);
\draw(-1,-8)--(-2,-9)node[below left]{$q$};

\draw[red] (9,-9)--++(1,-1);
\draw[red] (8,-10)--++(1,-1);
\draw[red] (7,-11)--++(1,-1);
\draw[red] (5,-13)--++(1,-1);

\draw[green!60!black] (4,-13)--++(-1,-1);
\draw[green!60!black] (3,-12)--++(-1,-1);
\draw[green!60!black] (2,-11)--++(-1,-1);
\draw[green!60!black] (1,-10)--++(-1,-1);

\node at (9,-9)[circle,inner sep=1.2pt,fill]{};
\node at (8,-9)[circle,inner sep=1.2pt,fill]{};
\node at (7,-9)[circle,inner sep=1.2pt,fill]{};
\node at (6,-9)[circle,inner sep=1.2pt,fill]{};

\draw[dashed] (11,-9.5)--++(2,0)node[right]{$t$}; \draw[dashed]
(6,-14.5)--++(2,0)node[right]{$u$}; \draw[dashed]
(-3,-8.5)--++(-2,0)node[left]{$s$};
\end{tikzpicture}
\caption{The key step in the proof of \cref{lower}.  The paths of $H'$ must
all cross path $n$ before time $s$, therefore at least as many paths must
cross path $p$ during the same time interval.} \label{interval}
\end{figure}
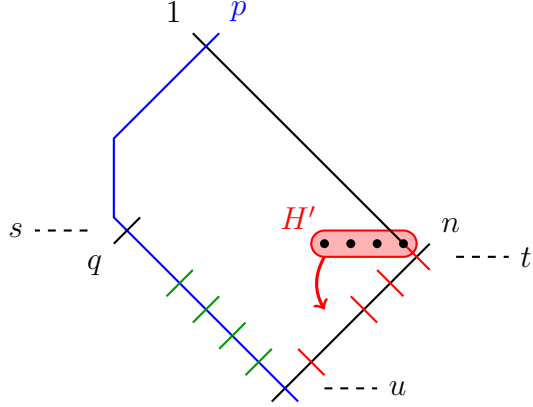

\sloppypar Finally, the cases $\pi_{s}=[\ldots, 1,
\ldots,n,\ldots,p,q,\ldots]$ and $\pi_{s}=[\ldots, n,
\ldots,1,\ldots,p,q,\ldots]$ are impossible, since together with our
assumption about $\pi_t$ they imply contradictions to simplicity.

Now we count moves.  There at least $r-3$ terrible blocks, each of which has
at least $r-1$ bad paths, which have at least $3$ moves, so the total number
of moves is at least
$$3(r-3)(r-1) \geq 3(r^2 - 4r) \geq 3n-c\sqrt n$$
for some $c>0$.

For general $n$, we use the same construction with $r=\lfloor
\sqrt{n-2}\rfloor$, add an extra $n-r^2-2< 2\sqrt n +1$ elements at the end
of the permutation, and adjust the constant.
\end{proof}

It is tempting to try to extend the ideas of the above proof to show that
there are permutations for which any simple tangle has $\gg n$ moves as
$n\to\infty$ (perhaps even $\Omega(n \log n)$).  A candidate permutation
might be constructed recursively: a ``level-$k$ permutation'' would have the
same structure as $\pi$ above, except with each block replaced with a smaller
level-$(k-1)$ permutation; the number of levels might be chosen to be of
order $\log n$ (or at least something $\gg 1$).  We have not succeeded in
completing such an argument.  Indeed, we do not know whether in fact $O(n)$
moves (or even $O(1)$ moves per path) suffice for a simple tangle.

\section{Per path lower bound}
\label{s:path-lower}

Finally, we prove a lower bound on moves per path that applies even for
non-simple tangles, as mentioned in the introduction.

\begin{prop}\label{3-moves}
For any $n>8$ there exists a permutation $\pi \in S_n$ such that any tangle
performing $\pi$ has a path with at least 3 moves.
\end{prop}

Our proof of this seemingly simple statement is surprisingly intricate, and
involves the two lemmas below.  The permutation will be
$$\pi := \bigl[n, \;\; 3, 2, \;\; n-3, n-4,
\dots, 5, 4,  \;\; n-1, n-2,\;\; 1\bigr].$$
 Recall that a pair of elements $i,j$ is said to be an
 inversion of a permutation $\pi$ if $i<j$ but $\pi^{-1}(i)>\pi^{-1}(j)$.

\begin{lemma} \label{lm:rev-crosses}
Let $T$ be a tangle performing any permutation of the form $\pi=[n,\ldots,1]$
with each path making at most $2$ moves. Let $1 <i < j< n$. If $i,j$ is an
inversion then paths $i$ and $j$ cross each other exactly once. If $i,j$ is
not an inversion then paths $i$ and $j$ either do not cross or cross exactly
twice. In the latter case, the permutation at the time $t$ just before paths
$1$ and $n$ cross is of the form $\pi_t = [\ldots, j, \ldots, 1,n, \ldots, i,
\ldots]$.
\end{lemma}

\begin{proof}
Paths $i$ and $j$ must cross an odd number of times if $i,j$ is an inversion,
and an even number of times if not.  Since each path has at most $2$ moves,
they cannot intersect more than twice.

Suppose that paths $i$ and $j$ cross twice.  Then $i$ must have an R-move
followed by an L-move, and vice-versa for $j$.  Since any path other than $1$
and $n$ must cross path $1$ during an L-move and cross path during an R-move,
the claimed form of $\pi_t$ follows.
\end{proof}

\begin{lemma} \label{lm:dx} Let $T$ be a tangle performing a permutation of the form $\pi=[n,\ldots,1]$
with each path making at most $2$ moves.  Let $z < a < b$ be some paths of
$T$ that first cross $n$ and then cross $1$.  If path $z$ crosses neither $a$
nor $b$, then $a$ and $b$ do not cross each other.
\end{lemma}

\begin{figure}
\centering
\begin{tikzpicture}[scale=.27]
  \fill[blue!10] (0,0)--(3,3)--(3,6)--(6,9)--(14,1)--(6,-7)--(4,-5)--(4,-4);
  \draw[thick] (-1,-1)node[below]{$n$}--(3,3)--(3,6)--(7,10)node[above]{$n$};
  \draw[thick] (-1,1)node[above]{$1$}--(4,-4)--(4,-5)--(7,-8)node[below]{$1$};
  \draw[thick] (5,10)node[above]{$a$}--(15,0)node[below]{$a$};
  \draw[thick] (5,-8)node[below]{$b$}--(15,2)node[above]{$b$};
  \fill[blue] (0,0) node[left=5pt,blue]{$W$}circle (.25);
  \fill[blue] (14,1)node[right=5pt,blue]{$E$} circle (.25);
  \fill[blue] (6,-7)node[right=5pt,blue]{$S$} circle (.25);
  \fill[blue] (6,9)node[left=5pt,blue]{$N$} circle (.25);
  \draw[orange,thick] (2,2)--(6,-2)--(6,-5)--(5,-6);
  \draw[green!70!black,thick] (4,7)--(12,-1);
  \draw[green!70!black,thick] (2,-2)--(9.5,5.5);
  \draw[red,thick] (12,5)--(10,3)--(14,-1);
\end{tikzpicture}
\caption{Illustration of the proof of \cref{lm:dx}: the region formed by paths $a,b,1,n$,
and the four types of path that may intersect it.
An RL path necessitates an LR path, but an LR path requires two further moves
in order to cross paths $n$ and $1$.}\label{trouble}
\end{figure}
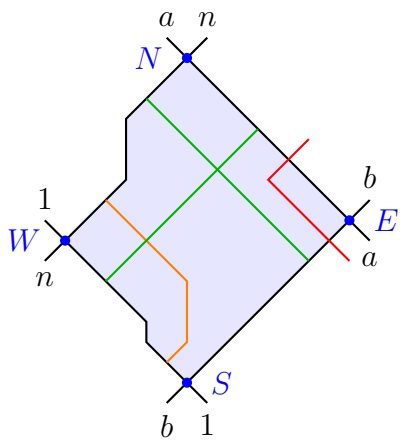
\begin{proof}
Suppose on the contrary that paths $a$ and  $b$ cross.  By
Lemma~\ref{lm:rev-crosses}, they cross only once.  Path $b$ cannot cross $a$
before $n$, since then $b$ would have more than $3$ moves.  Similarly, path
$b$ cannot cross $a$ after $1$, since $a$ would have $3$ moves.

Therefore, path $b$ crosses $n$, then $a$, then $1$.  Let $N,E,S,W$ be the
intersection points of the pairs of paths $(n,a),(a,b),(1,b),(1,n)$
respectively, all of which are unique by \cref{lm:rev-crosses}.  These points
are connected in clockwise order by four portions of the paths $a,b,1,n$,
which bound a region $NESW$.  See \cref{trouble}.  Note that any path other
than $1$ or $n$ has exactly one L-move and one R-move.  Therefore, the sides
$NE$ and $ES$ (which form part of paths $a$ and $b$) are straight line
segments. On the other hand, the sides $SW$ and $WN$ may each contain at most
one vertical segment, since paths $1$ and $n$ may have two moves in the same
direction separated by a vertical segment.

Let $\ell(SW)$ denote the number of intersections of the side $SW$ with paths
other than $1,n,a,b$ (which corresponds to the length of its non-vertical
portions), and similarly for each of the other three sides.  By the above
observations,
$$\ell(WN)\leq\ell(ES);\qquad \ell(SW)\leq \ell(NE).$$
Every path other than $1,n,a,b$ than intersects $NESW$ must do so either in a
single L-move or R-move, or with an L-move followed by an R-move, or
vice-versa.  Let $p(L),p(R),p(LR),p(RL)$ denote the numbers of paths in each
category.  We have
\begin{align*}
\ell(WN)&=p(RL)+p(R);&\quad
\ell(NE)&=p(LR)+p(L);\\
\ell(ES)&=p(LR)+p(R);&\quad
\ell(SW)&=p(RL)+p(L).
\end{align*}
Combining these equations with either of the above inequalities gives
$$p(RL)\leq p(LR).$$
We have $p(RL)\geq 1$, because of path $z$.  However, $p(LR)\geq 1$ gives a
contradiction, because such a path has at least $4$ moves, in order to cross
$n$, $a$, $b$ and $1$.
\end{proof}

\begin{proof}[Proof of \cref{3-moves}]
Consider $$\pi := \Bigl[n, \;\; \underbrace{3, 2}, \;\; \underbrace{n-3, n-4,
\dots, 5, 4},  \;\; \underbrace{n-1, n-2},\;\; 1\Bigr].$$ We denote $A = \{2,
3\}$, $B= \{4,\dots, n-3\}$, and $C= \{n-2, n-1\}$.

Suppose for a contradiction that there exists a tangle $T$ performing $\pi$
in which each path has at most $2$ moves. First suppose that $T$ is simple.
In each permutation of the tangle, the elements of $A$ precede the elements
of $B$, which precede the elements of $C$.  Let $t$ be the time of the swap
$1,n$, and suppose that some element $x$ appears to the right of this swap,
i.e.\ $\pi_t = [\ldots, 1, n, \ldots, x, \ldots]$.  If $x\in A\cup B$ then
paths $x<n-1<n-2$ contradict \cref{lm:dx}.  Thus $x\in C$.  Similarly, if
$\pi_t = [\ldots,y,\ldots, 1, n, \ldots]$ then $y\in A$.  Thus there is no
possible location for the elements of $B$ in $\pi_t$, a contradiction.

Suppose on the other hand that $T$ is not simple.  Thus there exist paths
$i,j$ that \df{double-cross} (i.e.\ have two crossings).  By
\cref{lm:rev-crosses}, the pair $i,j$ is not an inversion, therefore $i,j$
are from two different sets among $A,B,C$.  We claim that there exist $i' \in
A$ and  $j' \in C$ whose paths double-cross. Suppose not. Without loss of
generality, assume that $i \in A$ and $j \in B$ double-cross.  Since path $i$
and any path of $C$ do not double-cross, they do not cross at all by
Lemma~\ref{lm:rev-crosses}.  Since path $i$ has an R-move then an L-move, we
have $\pi_t = [\ldots, 1, n, \ldots, i,\ldots]$.  So paths $i<n-1<n-2$
contradict Lemma~\ref{lm:dx}. Thus $i',j'$ exist as claimed.

Since $n > 8$ and $|B| > 2$, there are at least two elements $u,v$ of $B$
that either both cross $1$ before $n$, or both cross $n$ before $1$.  Without
loss of generality, assume the latter.  Since paths $i'$ and $u$ both move
right then left, they cannot double-cross, and therefore by
\cref{lm:rev-crosses}, they do not cross.  By the same reasoning, $i'$ and
$v$ do not cross.  But now the paths $i'<u<v$ give a contradiction to
\cref{lm:dx}.
\end{proof}

\section*{Acknowledgements}

We thank Omer Angel, Franz Brandenburg, David Eppstein,
Martin Fink, Michael Kaufmann, Peter Winkler and Alexander
Wolff for valuable conversations.

\section*{Open problems}

\begin{enumerate}[1.]
\item What is the asymptotic behavior as $n\to\infty$ of the maximum over
    permutations $\pi\in S_n$ of the minimum number of moves among {\em
    simple} tangles that perform $\pi$?  In particularly, is it $O(n)$?
    Our results show only that it is between $3n-o(n)$ and $O(n\log n)$.
\item Similarly, what is the asymptotic behavior of the number of moves
    in the worst path (again, for the best simple tangle performing the
    worst permutation)?  Our bounds are $3$ and $O(\log n)$.
\item For general (not necessarily simple) tangles, what is the smallest
    constant $a$ for which there exists a tangle with at most $an$ moves
    for every permutation in $S_n$, for every $n$?  And what is the
    smallest $b$ for which we can achieve at most $b$ moves per path?  We
    know that $2\leq a\leq 4$ and $3\leq b\leq 5$.
\item Many natural questions arise concerning permutations that can be
    performed by tangles of various restricted types.  For example,
    suppose that the swaps of a tangle occupy all even locations in a
    simply connected region bounded above and below by interfaces
    consisting of North-East and South-East steps, and on the left and
    right by interfaces of South-West and South-East steps, as in
    \cref{splat}. Note in particular that there is one cluster, and no
    ``holes''.  It is not difficult to show that any permutation can be
    performed by such a tangle of depth at most $O(n^2)$ (see
    \cref{splat-alg} for the idea), but this seems far too large.  What
    is the minimum depth needed?  Is there a simple characterization of
    the set of permutations that can be performed if the depth is
    restricted to be at most $n$ (say)?
\end{enumerate}
\begin{figure}
\centering
\begin{subfigure}[c]{.45\textwidth}\centering
\includegraphics[width=.4\textwidth]{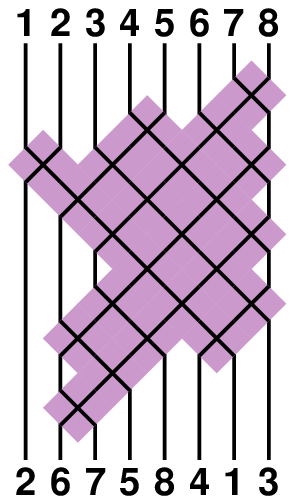}
\caption{A tangle occupying a simply connected region bounded by monotone interfaces,
as discussed in open problem 4.}
\label{splat}
\end{subfigure}
\hfill
\begin{subfigure}[c]{.45\textwidth}\centering
\includegraphics[width=.4\textwidth]{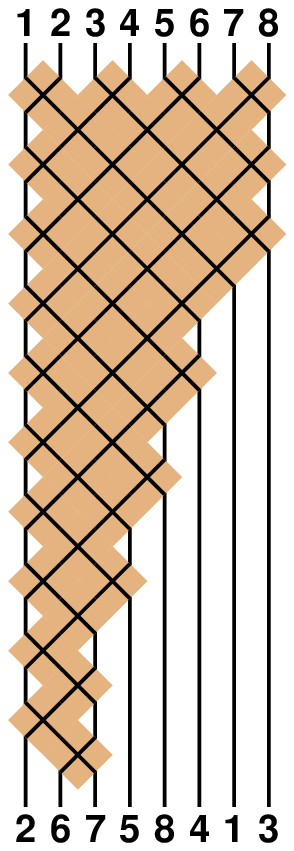}
\caption{A greedy construction of such a tangle: we apply alternate
rows of swaps in odd and even positions until path $\pi(n)$ is in the
rightmost position, then continue in the same way with locations $1,\ldots,n-1$.}
\label{splat-alg}
\end{subfigure}
\caption{}
\end{figure}

\bibliographystyle{abbrv}
\bibliography{tangles}
\end{document}